\documentclass[a4paper]{scrarticle}
\usepackage{amscd,amssymb,amsmath,amsthm}
\usepackage{graphicx}
\usepackage[backend=bibtex,style=ieee,giveninits=true,sorting=anyt]{biblatex}
\usepackage{hyperref}
\usepackage{color}
\usepackage{todonotes}
\usepackage[lofdepth,lotdepth]{subfig}
\newcommand{\Z}{\mathbb{Z}}
\newtheorem{thm}{Theorem}
\newtheorem{defn}{Definition}
\newtheorem{lemma}{Lemma}
\newtheorem{pro}{Proposition}
\newtheorem{rk}{Remark}
\newtheorem{cor}{Corollary}

\newcommand{\footnoteremember}[2]{
	\footnote{#2}
	\newcounter{#1}
	\setcounter{#1}{\value{footnote}}
}
\newcommand{\footnoterecall}[1]{
	\footnotemark[\value{#1}]
}
\begin{document}
\title{Coexistence of 	
localized Gibbs measures and delocalized gradient Gibbs measures on trees}	
\author{
	Florian Henning,\footnoteremember{RUB}{Ruhr-Universit\"at   Bochum, Fakult\"at f\"ur Mathematik, D44801 Bochum, Germany} \footnote{Florian.Henning@ruhr-uni-bochum.de,
	\newline \url{https://www.ruhr-uni-bochum.de/ffm/Lehrstuehle/Kuelske/henning.html}}
		\and Christof K\"ulske \footnoterecall{RUB} \footnote{Christof.Kuelske@ruhr-uni-bochum.de, 
			\newline
		\url{https://www.ruhr-uni-bochum.de/ffm/Lehrstuehle/Kuelske/kuelske.html}}\, 
	\,  
}
	
	\maketitle
	
\begin{abstract}
We study gradient models for spins taking values in the integers (or an integer lattice), 
which interact via a general potential depending only on the differences 
of the spin values at neighboring sites, located on a regular tree with $d+1$ neighbors. 
We first provide general conditions in terms of the relevant 
$p$-norms of the associated transfer operator $Q$
which ensure the existence of a countable family of 
proper Gibbs measures, describing localization at different heights.
Next we prove existence of delocalized gradient Gibbs measures, under natural 
conditions on $Q$. We show that the two conditions can be fulfilled at the same time, which then implies coexistence of both types of measures 
for large classes of models including the SOS model, and heavy-tailed models 
arising for instance for potentials of logarithmic growth. 
\end{abstract}

\section{Introduction}
Random fields with gradient interactions have been studied on graphs with various 
geometries, foremost on the lattice, but also on different graphs, like infinite trees \cite{FS97},\cite{CoDeuMu09},\cite{DeuGiIo00},\cite{KoL14},\cite{BK94},\cite{HKLR19}.

In the present chapter we look at $\mathbb{Z}^k$-valued spin variables $\sigma_x$ 
located on the vertices $x$ of  a $d$-regular tree, where each site has $d+1$ neighbors,  
with gradient interaction given by an even function $U:\mathbb{Z}^k \rightarrow \mathbb{R}$.  
The Hamiltonian of such a model becomes 
\begin{equation*}
\sum_{x \sim y}U(\sigma_x-\sigma_y),
\end{equation*}
where the sum runs over pairs of nearest neighbors on the tree.
We may assume without loss that $U(0)=0$. 
As there are no Gibbs measures on the line $d=1$, we assume that $d\geq 2$, 
with a particular interest in the case of large $d\uparrow \infty$. 

Our first aim is to prove existence and localization properties of corresponding  
Gibbs measures, under minimal growth assumptions on the interaction.   
In concrete examples we will be interested mainly in the 
case of local spin space dimension $k=1$, but our method of proof and our general estimates of Theorems \ref{Coexthm: main} and \ref{thm: LargeDegree} 
work equally well for the case of higher $k$. 

Gibbs measures on trees have been mostly studied 
for finite local spin spaces, including the  
Ising model, the Potts models, 
and the discrete Widom-Rowlinson model. 
The translation invariant (splitting) Gibbs measures can be described 
in terms of the roots of polynomials in many such cases.  
Our present problem is more difficult as it is infinite-dimensional 
even for the tree-automorphism invariant states, so we can not 
hope for explicit solutions in the general case. There may be 
no solutions at all, due to non-compactness of the local state space. 
We will construct our measures via fixed points in suitable $l^p$-balls 
in regimes of sufficiently strong confinement, via a new contraction method. 
Some readers may feel reminded of Dobrushin uniqueness  
which is a different contraction method 
(cp. \cite{DoPe83} and Theorem 8.7 in \cite{Ge11})
but this would work only in the additional presence of a confining single-site potential which we don't have, and prove uniqueness
of the Gibbs measure. 
Dobrushin uniqueness makes no assumptions on the structure of the underlying 
graph other than upper bounds on the number of neighbors, but 
works only in regimes of weak interactions (small inverse temperatures)  between the variables, 
showing near independence of variables at different sites.
Our method 
is based on the description of Gibbs measures on trees via the so-called boundary law equation introduced by Zachary \cite{Z83}. 
It works in suitable regimes of strong interactions (corresponding to 
large inverse temperatures). Intuitively this 
means that the influence  
from the variables at the neighbors to the variable at a given vertex needs to be strong enough to localize its distribution also in infinite volume, for which 
large degrees $d$ are beneficial. 
The method yields an easy to handle model-independent 
criterion for existence of countably many different localized states    
(as there are countably many possible localization centers in the local state space).

In the second main part we turn to the existence of \textit{delocalized gradient} Gibbs measures.  
Gradient Gibbs measures (GGM) play an important role in mathematical physics and probability 
in the description of interfaces \cite{FS97}, \cite{Sh05}, \cite{EnKu08}, \cite{CKu12}. 
Gradient measures only describe increments, and do not carry information about the absolute height, 
as opposed to the proper Gibbs measures. If they are \textit{delocalized} it means that the gradient description is the only possible one: In this case there is no proper 
Gibbs measure which would carry information also about the absolute height, whose gradient distribution coincides with the delocalized gradient measure. In the present chapter we will show that the existence problems for 
both fundamentally different types of measures nevertheless allow for a unified treatment. 
For this we construct a good set $G_d$ with relevance for both problems, but where
different norms have to be taken.
As a consequence of our approach, we obtain coexistence regions 
for both types of measures. 

We will construct GGMs which depend on an internal parameter, the \textit{period}  
$q=2,3,4,\dots $. This provides us again, as in the case of proper Gibbs measures, 
with a countable family of measures, but for the GGMs this family is indexed by $q$. 
The reason for the occurrence of the family of GGMs is 
much different than in the case of the proper Gibbs measures, 
as GGMs for different $q$ are structurally different (see Corollary \ref{cor: Period}).

On the other hand, we also provide examples for the surprising case 
where localized Gibbs measures do exist, 
but (height period-$q$) delocalized gradient Gibbs measures can not exist. 
The construction of the latter involves nonsummable transfer operators
on trees of sufficiently large degrees. 
\subsection{Results}

It is convenient to describe the model equivalently in terms of the so-called 
\textbf{transfer operator}  $Q$ which associates to an increment 
$\sigma_x-\sigma_y=j\in \mathbb{Z}^k$ along the oriented edge $(x,y)$ 
of the graph, the weight \[Q(j)=e^{-\beta U(j)}.\]  
Clearly, logarithmic growth of $U$ as a function of the increment then corresponds to polynomial decay of $Q$.

\subsubsection{Existence of localized Gibbs measures for gradient models}

\textbf{Main existence theorem, possibly heavy-tailed transfer operators.}  
As our first main general result, we prove in Theorem  \ref{Coexthm: main}
the existence of tree-automorphism invariant states concentrated around any given $i\in \mathbb{Z}^k$, 
under the assumptions on the interaction which should be viewed as 
finiteness of the $\frac{d+1}{2}$-norm of $Q$, 
and smallness of its $d+1$-norm on $\mathbb{Z}^k \setminus \{0\}$. 
More precisely,  
the existence results hold whenever these two norms 
lie in a certain two-dimensional "good" set $G_{d}$ which depends only on the degree $d$ of the tree.  
This allows for heavy-tailed $Q$, but includes also cases of faster decay (e.g. SOS model or the discrete Gaussian). Note that heavy-tailed $Q$ trivially implies 
heavy-tailed single-site marginals for any Gibbs measure.

As a further consequence of our method we also obtain 
the quantitative localization bounds of Theorem \ref{Coexthm: main} on the single-site marginal of the measure. 

%

\textbf{Existence of distinct localized Gibbs measures for $\mathbb{Z}_q$-clock models.} 
With our method we may 
also treat finite-spin $\mathbb{Z}_q=\{0,1,\dots, q-1\}$-valued models with discrete rotation symmetry beyond the cases of explicitly solvable 
boundary law equations. This is relevant for  clock-models, but will also be
relevant for Subsection \ref{ExDelocGrad}. 
We prove the existence of $q$ distinct ordered states, which are localized around a given spin value, together with quantitative control of the localization.  
This implies low-temperature ordering for general classes of 
models (for which the Potts-model and the discrete rotator model with scalar-product interaction are just special examples).  
The precise assumption on the interaction needed involves the analogous 
pair of $p$-norms as for the Gibbs measures of the gradient models,  
but on the finite spin space. 

\textbf{Ideas of proof.}
The method of proof for the gradient model is based on a study of Zachary's fixed point 
equation for boundary laws 
(which are positive measures on the single-spin space $\mathbb{Z}^k$) in a suitable $l^p$-space
via a contraction argument. 
For each choice of spin value $i\in \mathbb{Z}^k$ 
we construct a boundary law solution on a suitable
ball of boundary laws concentrated around $i$. 
The optimal choice of the exponent $p$ is determined
by Zachary's summability condition (which is explained by the 
requirement to have summable 
single-site marginals for the infinite-volume Gibbs 
measure corresponding to the boundary law). 
Given that, 
the choice of the exponents of the two different 
$Q$-norms in the hypothesis of our theorem is again optimal, 
and explained from Young's convolution inequality which appears in the proof. We do not need to assume convexity of the interaction. 
In this way we get 
non-uniqueness of Gibbs measures via uniqueness of localized boundary law solutions in $l^p$-balls. 
This general setup turns out to be useful as it allows to approach also the existence problems for gradient Gibbs measures below.

\subsubsection{Existence of delocalized 
	gradient Gibbs measures} \label{ExDelocGrad}

Recall that a gradient specification may admit gradient Gibbs measures 
even when it does not admit proper Gibbs measures.
Examples of such gradient Gibbs measures for tree models described by $Q$ 
have been discussed and constructed in \cite{KS17} and \cite{HKLR19}. 
We consider gradient Gibbs measures which can be constructed 
via $q$-height-periodic boundary laws for fixed height period $q$ (see Section \ref{Sec: A two-layer construction of gradient Gibbs measures}). 
While some very specific small-$q$
examples were already constructed in the case of the SOS model in \cite{HKLR19} via 
explicit solutions of polynomial equations, we 
are aiming here for a general existence theory allowing also for arbitrarily large $q$.  

\textbf{Existence.} In  Theorem \ref{thm: main2} we first provide 
a uniform-$q$ existence result in terms of the $1$-norm of the transfer operator $Q$. 
In the second part we state that the very same condition of Theorem \ref{Coexthm: main}
involving the good region $G_d$ (under an additional summability assumption),   
also ensures existence for large-enough periods $q$. 
A key idea of proof for this is to use continuity of the existence criterion derived
for the Gibbs measures at period $q=\infty$.

In this case we have found a 
coexistence region for localized Gibbs and delocalized gradient Gibbs measures.   
However, there is also another interesting regime 
for potentials of slow growth, see Subsection \ref{SSSDrei}: 
For large $d$ the good region $G_d$ may extend into the region of infinite $1$-norm.  
If this is the case, our results imply that localized Gibbs measures exist, but gradient Gibbs measures do not exist for any height-period $q$. 

\textbf{Localization vs. delocalization, two-layer construction of gradient measures.}
Theorem \ref{thm: deloc} explains the difference between proper localized 
Gibbs states, and the height-period $q$ gradient Gibbs measures, via 
properties of both types of measures restricted to a semi-infinite path on the tree. 
The corresponding random walk path localizes and has an invariant probability 
distribution for the Gibbs measures, while the random walk delocalizes for the 
height-period $q$ gradient measures.   
In the context of these statements 
we also provide a new and  useful view to the gradient Gibbs measures 
in terms of a two-layer hidden Markov model construction which is very intuitive and 
interesting in itself.   

\textbf {Identifiability.} We show under no further assumptions
that different height-periods and different boundary 
law solutions (modulo height-shift) lead to different gradient Gibbs states, 
using ergodicity. 
This extends previous partial identifiability results of \cite{HKLR19} obtained by algebraic 
arguments.

\subsubsection{Applications: Binary tree, large degree asymptotics, SOS model, 
	log-potential}\label{SSSDrei}

We discuss our general existence results of Theorem  \ref{Coexthm: main} and Theorem \ref{thm: main2}, in 
more detail for the binary tree, and in the limit of large degrees $d\uparrow \infty$. 
On the binary tree we obtain an explicit form of the boundary curves of the 
good region $G_2$ ensuring 
existence of proper Gibbs measure, see 
Figure \ref{fig: GoodSet} and Proposition \eqref{pro: GoodBinary}.

For large degrees $d$, discussing the asymptotics of the region $G_d$, 
we show the existence of localized states for $\beta \in (\beta(d),\infty)$ with the model-independent form of the asymptotics 
$\beta(d)\sim \frac{\log d}{d}\downarrow 0$ as $d$ gets large, 
see Theorem \ref{thm: LargeDegree}. 
We also illustrate our general estimates with the examples of two potentials, the SOS model with potential
\[U(j)= |j|, \] and the log-potential with \[U(j)= \log (1+|j|), \]
for various degrees $d$, for which we provide explicit numbers. 

The chapter is organized as follows.  We present our results on Gibbs measures in Section \ref{sec: Existence of localized Gibbs measures}, 
on gradient Gibbs measures in Section \ref{sec: Existence of delocalized gradient Gibbs measures}, and discuss applications in Section \ref{Coexsec: Applications}. The proofs
are found in Section \ref{Coexsec: Proofs}.
\subsection*{Acknowledgements} The authors thank an anonymous referee for clarifying remarks
and suggestions.\\ Florian Henning is partially supported by the Research Training Group 2131 \textit{High-dimensional phenomena in probability - Fluctuations and discontinuity} of the German Research Council (DFG).
\section{Existence of localized Gibbs measures} \label{sec: Existence of localized Gibbs measures}	
\subsection{Definitions}
We consider a class of models with an integer-lattice $\mathbb{Z}^k, k \geq 1$, as local state space and the Cayley-tree $\Gamma^d=(V,L)$ of order $d \geq 2$ (i.e., the $d$-regular tree) as index set and denote the configuration space $(\mathbb{Z}^k)^V$ by $\Omega$. 

More precisely, the Cayley tree $\Gamma^d$ is an infinite tree, i.e., a locally finite connected graph without cycles, such that
exactly $d+1$ edges originate from each vertex.
Two vertices $x,y \in V$ are called \textit{nearest neighbours} if
there exists an edge $l \in L$ connecting them. We will use the
notation $l=\{ x,y\}$. A collection of nearest neighbour
pairs $\{ x,x_1\}, \{ x_1,x_2\},...,\{
x_{n-1},y\}$ is called a \textit{path} from $x$ to $y$. The
distance $d(x,y)$ on the Cayley tree is the number of edges of the
shortest path from $x$ to $y$.
In contrast to the set of unoriented edges $L$ we also consider the set of oriented edges $\vec{L}$. An \textit{oriented edge pointing from $x$ to $y$ } is simply the pair $(x,y)$ of vertices $x,y \in V$. 
Furthermore, for any $\Lambda \subset V$ we define its outer boundary by
\begin{equation*}
\partial \Lambda := \{ x \notin \Lambda : d(x,y) = 1 \mbox{ for some } y \in \Lambda\}.
\end{equation*}
Now assume that we have a nearest-neighbour interaction potential $\Phi$ with corresponding strictly positive \textit{transfer operator} $Q$ defined by 
\begin{equation*}
Q_b(\zeta):=\exp\left(-\Phi_b(\zeta) \right)>0
\end{equation*}
for any edge $b=\{x,y\} \in L$ and $\zeta \in \mathbb{Z}^b$.

The kernels of the Gibbsian specification $(\gamma_\Lambda)_{\Lambda \subset \subset V}$ then read 
\begin{equation} \label{CoexDef: Gibbs specification}
\gamma_\Lambda(\sigma_\Lambda=\omega_\Lambda \mid \omega)=Z_\Lambda(\omega_{\partial \Lambda})^{-1} \prod_{b \cap \Lambda \neq \emptyset}Q_b(\omega_b).
\end{equation} 
Hence the family $(Q_b)_{b \in L} $ is required to that for any finite subvolume $\Lambda \subset V$ and any configuration $\omega \in V$ the partition function $Z_\Lambda(\omega)=Z_\Lambda(\omega_{\partial \Lambda})$ is a finite positive number (cp. condition (3.1) in \cite{Z83}).

In this chapter we focus on the special case of tree-automorphism invariant symmetric gradient interactions normalized at $0$, i.e., for any edge $b=\{x,y \} \in L$
\[Q_b(\omega_x,\omega_y)=Q(\omega_x-\omega_y)=\exp(-\beta U(\omega_x-\omega_y)),\] 
where the parameter $\beta>0$ will be regarded as inverse temperature and \\$U: \mathbb{Z}^k \rightarrow [0,\infty)$ is a symmetric function with $U(0)=0$.
\begin{defn}
	Let $S=\mathbb{Z}^k, \, \mathbb{Z}^k \setminus \{0\}, \, \mathbb{Z}_q \text{ or } \mathbb{Z}_q \setminus \{0\}$.
	
	For any $1 \leq p < \infty $ consider the Banach space
	\begin{equation*}
	l_{p}(S):=\big\{ x \in \mathbb{R}^S \mid \Vert x \Vert_{p,S}:=\big(\sum_{j \in S}  \vert x(j) \vert^p \big)^{\frac{1}{p}} < \infty  \big\}.
	\end{equation*}
\end{defn}

\subsection{Well-definedness and the main result }
Using this notation we are able to state our first result, stating that finiteness of the $\frac{d+1}{2}$-norm of a transfer operator $Q$ ensures well-definedness of its asscociated Gibbsian specification.
\begin{lemma} \label{lem: existenceModel}
	If $\Vert Q \Vert_{\frac{d+1}{2},\mathbb{Z}^k}<\infty$ then the Gibbsian specification \eqref{CoexDef: Gibbs specification} is well defined, i.e., for any finite connected volume $\Lambda \subset V$ and any boundary condition $\omega_{\partial \Lambda}$ the partition function $Z_\Lambda(\omega_{\partial \Lambda})$ is a finite number.
\end{lemma}

For any integer $d\geq2$ define the \textit{good set}
\begin{equation} \label{eq: good set}
\begin{split}
G_d:=\{&(\gamma, \delta) \in (1,\infty) \times (0,\infty) \mid \text{There exists an } \varepsilon>0 \text{ such that } \cr 
&\delta+\gamma \varepsilon^d  \leq \varepsilon \quad \text{ and } \quad 2d\gamma \varepsilon^{d-1}+2d\delta \varepsilon^d<1\}.
\end{split}
\end{equation}
Using this notation, our main theorem reads the following.
\begin{thm} \label{Coexthm: main}
	Fix any integer $d \geq 2$. For any strictly positive transfer operator $Q$ with $Q(0)=1$ set $\gamma:=\Vert Q \Vert_{\frac{d+1}{2}, \mathbb{Z}^k}$ and $\delta:=  \Vert Q \Vert_{d+1, \mathbb{Z}^k \setminus \{0\}}$. 
	
	If $(\gamma, \delta) \in G_d$ then there exists a family of distinct tree-automorphism invariant Gibbs measures $(\mu_i)_{i \in \Z^k}$ which are equivalent under joint translation of the local spin spaces.
	
	Moreover, the single-site marginal of each $\mu_i$ satisfies the following localization bounds
	\begin{equation*}
	\left(\delta \frac{1-\delta\varepsilon(\gamma ,\delta)^d}{1+\gamma\varepsilon(\gamma,\delta)^{d-1}} \right)^{d+1} \leq \frac{\mu_i(\sigma_0 \neq i)}{\mu_i(\sigma_0=i)}\leq \left( \delta \frac{1+\delta\varepsilon(\gamma, \delta)^d}{1-\gamma\varepsilon(\gamma,\delta)^{d-1}} \right),^{d+1} 
	\end{equation*}
	where $\varepsilon(\gamma,\delta)$ denotes the smallest positive solution to the equation
	\begin{equation*}
	\varepsilon=\gamma \varepsilon^d+\delta .
	\end{equation*}
\end{thm}
\begin{rk}
	Theorem \ref{Coexthm: main} stays true if $\Z^k$ is replaced by the ring $\Z_q$ and $Q$ is an even function on $\Z_q$. Such models are called clock models (cp. \cite{PeSt99}, \cite{JaKu14}) and Theorem \ref{Coexthm: main} delivers the existence of ordered phases in this case.
\end{rk}
\subsection{Background on the relation between  Gibbs measures and boundary laws}
In this subsection we summarize the key ingredient in constructing Gibbs measures to Markovian specifications for tree-indexed models, the notion of a boundary law. The theory presented goes back to Zachary \cite{Z83}.

\begin{defn}\label{Coexdef:bl}
	A family of functions $\{ \lambda_{xy} \}_{( x,y ) \in \vec L}$ with $\lambda_{xy} \in [0, \infty)^{{\Z}^k}$ and $\lambda_{xy}  \not \equiv 0$ is called a \textbf{boundary law} for the transfer operators $\{ Q_b\}_{b \in L}$ if 
	\begin{enumerate} 
		\item for each $( x,y ) \in \vec L$ there exists a constant  $c_{xy}>0$ such that the \textbf{boundary law equation}
		\begin{equation}\label{Coexeq: bl}
		\lambda_{xy}(\omega_x) = c_{xy} \prod_{z \in \partial x \setminus \{y \}} \sum_{\omega_z \in \mathbb{Z}^k} Q_{zx}(\omega_x,\omega_z) \lambda_{zx}(\omega_z)
		\end{equation}
		holds for every $\omega_x \in \mathbb{Z}^k$ and 
		\item for any $x \in V$ the \textbf{normalizability condition}
		\begin{equation}
		\sum_{\omega_x \in \mathbb{Z}^k} \Big( \prod_{z \in \partial x} \sum_{\omega_z \in \mathbb{Z}^k} Q_{zx}(\omega_x,\omega_z) \lambda_{zx}(\omega_z) \Big) < \infty
		\end{equation}
		holds true.
	\end{enumerate}
\end{defn}
Note that in \cite{Z83} the functions $\lambda_{xy}$ are considered as equivalence classes of families of functions, two functions being equivalent if and only if one is obtained by multiplying the other one by a suitable edge-dependent positive constant. The more explicit definition above is based on the notation used in \cite{Ge11}. Following this notation it is convenient to choose the constants in such a way that the boundary law is \textit{normalized at} $0$, i.e., $\lambda_{xy}(0)=1$ for all $( x,y ) \in \vec{L}$. 

The following result of Zachary establishes a correspondence between the set of those Gibbs measures which are also tree-indexed Markov chains (see Definition(12.2) in \cite{Ge11}) and the set of normalizable boundary laws: 
\begin{thm}[Theorem 3.2 in \cite{Z83}] \label{Coex: BLMC}
	Let $(Q_b)_{b \in L}$ be any family of transfer-operators such that there is some $\zeta \in \Omega$ with $Q_{\{x,y \}}(i, \zeta_y)>0$ for all $\{x,y\} \in L$ and any $i \in \Z^k$.
	
	Then for the Markov specification $\gamma$ associated to $(Q_b)_{b \in L}$  we have:
	\begin{enumerate}
		\item Each boundary law $(\lambda_{xy})_{( x,y ) \in \vec{L}}$ for $(Q_b)_{b \in L}$ defines a unique tree-indexed Markov chain $\mu \in \mathcal{G}(\gamma)$ with marginals
		\begin{equation}\label{Coex: BoundMC}
		\mu(\sigma_{\Lambda \cup \partial \Lambda}=\omega_{\Lambda \cup \partial \Lambda}) = (Z_\Lambda)^{-1} \prod_{y \in \partial \Lambda} \lambda_{y y_\Lambda}(\omega_y) \prod_{b \cap \Lambda \neq \emptyset} Q_b(\omega_b),
		\end{equation}
		for any connected set $\Lambda \subset \subset V$ where $y \in \partial \Lambda$, $y_\Lambda$ denotes the unique $n.n.$ of $y$ in $\Lambda$ and $Z_{\Lambda}$ is the normalization constant which turns the r.h.s. into a probability measure.
		\item Conversely, every tree-indexed Markov chain $\mu \in \mathcal{G}(\gamma)$ admits a representation of the form (\ref{Coex: BoundMC}) in terms of a boundary law (unique up to a constant positive factor).
	\end{enumerate}
\end{thm}
\subsection{Setup for the fixed-point method}
In the case of tree-automorphism invariant gradient interaction potentials the boundary law equation \eqref{Coexeq: bl} simplifies to
\begin{equation} \label{eq: homBLeq}
\lambda(\cdot)=c\sum_{j \in \Z^k}Q(\cdot-j)\lambda(j)^d
\end{equation}
where $c>0$ is any constant. A short calculation then shows that in this case $\lambda$ is normalizable if and only if $\lambda \in l_{\frac{d+1}{d} }(\Z^k)$.
The normalized homogeneous boundary law equation then reads
\begin{equation} \label{eq: normhomBLeq}
\lambda(i)=  \left(\frac{Q(i)+\sum_{j \in \mathbb{Z}^k \setminus \{0\}} Q(i-j)  \vert \lambda(j) \vert}{1+\sum_{j \in \mathbb{Z}^k \setminus \{0\}}Q(j)  \vert \lambda(j) \vert   } \right)^d
\end{equation}
for $i \in \mathbb{Z}^k \setminus \{0\}$ and $\lambda(0)=1$.
Note that under the assumption $\Vert Q \Vert_{1,\Z^k}<\infty$ equation \eqref{eq: normhomBLeq} will always have the trivial solution $\lambda \equiv 1$. This solution, however, is not an element of the space $ l_{\frac{d+1}{d}} (\Z^k)$, i.e., will not lead to a finite measure.  

Going over to the $d$th root, equation \eqref{eq: normhomBLeq} is equivalent to
\begin{equation} \label{FPeq}
x(i)= \frac{Q(i)+\sum_{j \in \mathbb{Z}^k \setminus \{0\}}Q(i-j)  \, \vert x(j) \vert^d }{1+\sum_{j \in \mathbb{Z}^k \setminus \{0\}}Q(j)  \, \vert x(j) \vert^d }
\end{equation} 
for $i \in \mathbb{Z}^k \setminus \{0\}$ and $x(0)=1$.

Since a tree-automorphism invariant boundary law $\lambda=(\lambda(i))_{i \in \mathbb{Z}^k}$ is normalizable if and only if $\lambda \in l_{\frac{d+1}{d}}(\Z^k)$, the family $x$ pointwisely given by $x(i):= \lambda(i)^{\frac{1}{d}}$ corresponds to a normalizable boundary law if and only if $x \in l_{d+1}(\Z^k)$.
We want to describe the set of solutions to \eqref{FPeq} by the set of fixed points to the operator \\$T: l_{d+1}(\mathbb{Z}^k \setminus \{0\}) \rightarrow l_{d+1}(\mathbb{Z}^k \setminus \{0\})$ given by  
\begin{equation} \label{Def: Operator}
T(x)(i):=\frac{Q(i)+\sum_{j \in \mathbb{Z}^k \setminus \{0\}}Q(i-j)  \,  \vert x(j) \vert^d }{1+\sum_{j \in \mathbb{Z}^k \setminus \{0\}}Q(j)  \, \vert x(j) \vert^d }
\end{equation}
in the subset \begin{equation*}
D:=\{x \in l_{d+1}(\mathbb{Z}^k \setminus \{0\}) \mid  \, x(i)  \geq 0 \text{ for all } i \in \mathbb{Z}^k \setminus \{0\} \}.
\end{equation*}
Note that the condition $x(i) \geq 0$ for all $i \in \mathbb{Z}^k$ is automatically satisfied for any fixed-point of $T$ in $l_{d+1}( \mathbb{Z}^k \setminus \{0\})$.

First we have to verify that $T(l_{d+1}( \mathbb{Z}^k \setminus \{0\})) \subset l_{d+1}( \mathbb{Z}^k \setminus \{0\})$ i.e., that $T$ is indeed an operator from the Banach space $l_{d+1}( \mathbb{Z}^k \setminus \{0\})$ into itself. This is ensured by the following lemma.
\begin{lemma} \label{lem: ineq}
	Let $\Vert Q \Vert_{\frac{d+1}{2}, \Z^k}< \infty$. Then for any $x \in l_{d+1}(\mathbb{Z}^k)$ we have 
	\begin{equation} 
	\begin{split}
	\Vert T(x) \Vert_{d+1, \mathbb{Z}^k \setminus \{0\}} &\leq  \Vert Q \Vert_{d+1, \mathbb{Z}^k \setminus \{0\}}+\Vert Q \Vert_{\frac{d+1}{2}, \mathbb{Z}^k} \, \Vert x \Vert_{d+1,\mathbb{Z}^k \setminus \{0\}}^{d} \cr 
	&< \infty. 
	\end{split}
	\end{equation} 
\end{lemma}
\subsection{\texorpdfstring{A $T$-invariant set and Lipschitz-continuity}{A T-invariant set and Lipschitz-continuity}}
In this section we give a criterion based on the $\frac{d+1}{2}$-norm and the $d+1$-norm of a transfer operator $Q$ ensuring that it is a contraction mapping in a small neighborhood around zero. This will be a key ingredient in the proof of Theorem \ref{Coexthm: main} and explains the form of the good set.
\begin{pro} \label{thm: cond. Banach}
	Define $\gamma:=\Vert Q \Vert_{\frac{d+1}{2}, \mathbb{Z}^k }$ and 
	$\delta:=\Vert Q \Vert_{d+1, \mathbb{Z}^k \setminus \{0\}}$. 
	
	Suppose that $\epsilon>0$ satisfies the inequality 
	\begin{equation} \label{eq: Ball}
	\delta+\gamma \epsilon^d \leq \epsilon.
	\end{equation}
	Then, for   
	the closed $\varepsilon$-ball $B_{\varepsilon}(0) \subset  l_{d+1}(\mathbb{Z}^k \setminus \{0\})$ the following holds. 
	\begin{enumerate}
		\item $T(B_{\varepsilon}(0) \cap D) \subset B_{\varepsilon}(0) \cap D$ 
		\item $T_{\vert B_{\varepsilon}(0) \cap D}$ is Lipschitz-continuous with respect to the $d+1$-norm with constant 
		\begin{equation} \label{eq: L-const}
		L=	2d \bigl( \gamma \epsilon^{d-1} + \delta\epsilon^{d}  \bigr). 
		\end{equation}
	\end{enumerate} 	
\end{pro}	
\section{Existence of delocalized gradient Gibbs  measures} \label{sec: Existence of delocalized gradient Gibbs measures}
Existence of one \textit{Gibbs measure} 
means that the $\mathbb{Z}$ or $\mathbb{Z}^k$ symmetry is broken, so many Gibbs measures at different heights also then exist. We have studied this question in the previous section.  
\textit{Gradient Gibbs measures} never carry information about the heights, only about height differences, whether in the localized or the delocalized regime. 
In the localized regime the expected height difference between far-apart vertices remains uniformly bounded. In that case there are many different possible expected heights, all compatible with the same gradient Gibbs measure.
In \textit{delocalized} regimes (which we are studying in the present section) 
the height differences between heights at far-apart vertices diverge. 
In this section we will describe a general existence theory for $q$-periodic gradient Gibbs measures which are delocalized and structurally different for different $q$. 
The simplest example of this is the free state, corresponding to $q=1$. In the free state
all increment variables are i.i.d. with single-edge distribution given by normalizing 
$Q$.  For higher periods $q$ the states become non-trivial, and can be understood 
via a two-layer construction in which a $q$-state clock model appears as an internal building block, 
compare also Remark \ref{rk: ConstrGGM} and Corollary \ref{cor: Period}.
Explicit examples will be found for two specific models in Section \ref{Coexsec: Applications}. 
All states we discuss are tree-automorphism invariant, and in particular have zero tilt. 
Existence and properties of non-tree automorphism invariant states, 
with or without tilt, poses serious new challenges for future research beyond the present work.

\subsection{Height-periodic boundary laws and their associated gradient Gibbs measures - preliminaries}
In this section we restrict to the integers $\Z$ as local state space and deal with the case of spatially homogeneous \textit{height-periodic} boundary laws, i.e.,~elements of $(0, \infty)^\Z$ satisfying the boundary law equation \eqref{eq: homBLeq} which are additionally periodic. For this, we necessarily have to assume that $Q \in l_1(\Z)$. \\Writing $\mathbb{Z}_q=\{0, \ldots, q-1\}$ for the$\mod q$ residue class ring, any $q$-periodic ($q=2,3, \ldots$) solution $\lambda_q$ to the boundary law equation \eqref{eq: homBLeq}~is obtained as a solution to the following $q-1$ dimensional system of equations 
\begin{equation*}
\lambda_q(\bar{i})=\left(\frac{\sum_{\bar{j} \in \Z_q}Q_q(\bar{i}-\bar{j}) \lambda_q(\bar{j})}{\sum_{\bar{j} \in \Z_q}Q_q(\bar{j}) \lambda_q(\bar{j})} \right)^d, \quad \bar{i} \in \Z_q,
\end{equation*}
where $Q_q(\bar{j}):=\sum_{l \in \bar{j}}Q(l)$ for all $\bar{j} \in \Z_q$ and $\lambda_q(\bar{i}):=\lambda_q(i)$ for any $i \in \bar{i}$. 
As a height-periodic boundary law is not normalizable in the sense of Definition \ref{Coexdef:bl} there is no way of constructing a Gibbs measure from it. However, it is still possible to assign a probability measure on the space of increments $\Z^{\vec{L}}$ which is a \textit{gradient Gibbs measure} in the sense that it obeys a DLR-equation with respect to the kernels \ref{CoexDef: Gibbs specification}. (cp. Thm. 4.1 in \cite{KS17} and Thm. 3.8 in \cite{HKLR19}).

First note that in the case of a $q$-periodic boundary law $\lambda_q$ the function \begin{equation*}
P_q: \Z^2 \rightarrow [0,1] \, ; \, P_q(i,j):=\frac{Q(i-j) \lambda_q(j)}{\sum_{l \in \Z}Q(i-l) \lambda_q(l)}
\end{equation*} depends only on the increment $i-j$ and the$\mod q$ value of $j$ (or $i$ equivalently) thus $P_q$ can be considered as a real function $\bar{P}_q$ on $\Z_q \times \Z$ given by 
\[\bar{P}_q(\bar{i},j-i):=P(i,j). \]
This means that it describes a $q$-periodic environment for a random walk. More precisely, the following two-step procedure is done: First fix a path on the tree and perform a random walk on $\Z_q$ along the path, which will be referred to as the \textit{induced chain}, or \textit{fuzzy chain} with transition matrix \begin{equation*}
P'_q: \Z_q^2 \rightarrow [0,1] \, ; \, P'_q(\bar{i}, \bar{j}-\bar{i}):=\sum_{l \in \bar{j}-\bar{i}}\bar{P}_q(\bar{i}, l), \end{equation*}
i.e.,
\begin{equation*}
P_q'(\bar{i}, \bar{j})=\frac{Q_q(\bar{i}-\bar{j})\lambda_q(\bar{j})}{\sum_{\bar{s} \in \Z_q}Q_q(\bar{i}-\bar{s})\lambda_q(\bar{s})}.
\end{equation*}
In the second step a random walk on the integer-valued gradient variables along the path is performed conditional on the realization of the fuzzy chain.

Conditional on that the fuzzy chain has an increment $\bar{s} \in \Z_q$ along an edge,
the marginal probability distribution of increments along this edge is the $\lambda_q$-independent measure on $\Z$
\begin{equation} \label{Coexeq: RWPEtr}
\rho^Q_{q}(j \mid \bar{s})= \chi(j \in \bar{s})\frac{Q(j)}{Q_q(\bar{s})}.
\end{equation}
\subsection{A two-layer construction of gradient Gibbs measures} \label{Sec: A two-layer construction of gradient Gibbs measures}
In this way we obtain the following measure on the space of gradient configurations on the tree.
Let $\alpha$ denote the stationary distribution for the fuzzy chain given by $\alpha(i)=\frac{ \lambda_q(i)^\frac{d+1}{d}}{\Vert \lambda_q \Vert_\frac{d+1}{d}^\frac{d+1}{d}  }$, for $i \in \Z_q$.
Further consider any vertex $w$ on the Cayley tree and let $P^\text{f.c.}_{w,\bar{s}}$ denote the distribution of the tree-indexed fuzzy chain $(\sigma'_x)_{x \in V}$ on $\Z_q^V$ with transition matrix $P_q'$ and conditioned on $\sigma'_w=\bar{s}$.
Then the measure $\nu^{\lambda_q}$ on the space of gradient configurations $\Z^{\vec{L}}$ has finite-volume marginals given by
\begin{equation} \label{Coexeq: ConstructionOfGGM}
\begin{split}
&\nu^{\lambda_q}(\eta_\Lambda=\zeta_\Lambda) \cr 
&=\sum_{\bar{s} \in \Z_q} \alpha(\bar{s}) \sum_{\sigma_\Lambda' \in \Z_q^\Lambda }P^{\lambda_q,\text{f.c.}}_{w,\bar{s}}(\sigma'_\Lambda) \prod_{(x,y) \in \vec{L}, \, x,y \in \Lambda}\rho^Q_{q}(\zeta_{(x,y)} \mid \sigma'_y-\sigma'_x) \cr 
&=\sum_{\sigma'_\Lambda \in \Z_q^\Lambda}P^{\lambda_q, \text{f.c.}}(\sigma'_\Lambda)\prod_{(x,y) \in \vec{L}, \, x,y \in \Lambda} \rho^Q_{q}(\zeta_{(x,y)} \mid \sigma'_y-\sigma'_x) 
\end{split}
\end{equation}  
where $\Lambda \subset V$ is any finite set and $w \in \Lambda$.  

The measure $P^{\lambda_q, \text{f.c.}}$ is exactly the distribution of the tree-indexed Markov chain on $\Z_q^V$ associated to the boundary law $\lambda_q$ by the version of Theorem \ref{Coex: BLMC} for the finite local state space $\Z_q$.
\begin{rk} \label{rk: ConstrGGM}
	Note that we obtain the gradient measure $\nu^{\lambda_q}$ by sampling first the hidden fuzzy spin variables $\sigma'$ and then the increment variables $\eta$ according to \eqref{Coexeq: RWPEtr},  conditionally independent on $\sigma'$ over all edges. As both mechanisms are tree-automorphism invariant, also the tree-automorphism invariance of the gradient measure is immediate. 
\end{rk}
\begin{rk} 
	From full tree-automorphism invariance and symmetry of the underlying potential it follows that the gradient Gibbs measure $\nu^{\lambda_q}$ has zero tilt, i.e., $E_{\nu^{\lambda_q}}[\eta_{(x,y)}]=0$ for any $(x,y) \in \vec{L}$.
\end{rk}
In the first theorem of this section we will give some criteria ensuring existence of height-periodic boundary law solutions. Afterwards we will show that the associated gradient Gibbs measures are distinct from the gradient spin projections of the localized Gibbs measures given by Theorem \ref{Coexthm: main}.
\subsection{Existence of gradient Gibbs measures}
The existence criterion for a countable family of gradient Gibbs measures indexed by $q$ involves the same good set $G_d$ as for the Gibbs measures (see Theorem \ref{Coexthm: main}).
\begin{thm} \label{thm: main2}
	Fix any integer $d \geq 2$. Let $Q \in l_1(\Z)$ be any strictly positive transfer operator with $Q(0)=1$. Then the following holds true.
	\begin{enumerate}		
		\item If $(\Vert Q \Vert_{1, \mathbb{Z}}, \Vert Q \Vert_{1, \mathbb{Z}\setminus \{0\}}) \in G_d$ then for any $q \geq 2$ there exist 
		tree-auto-morphism invariant GGMs coming from $q$-periodic boundary law solutions which are not equal to the free state. 
		\item Further set $\tilde{Q}(i):=\sup_{\vert j \vert  \geq \vert i \vert }Q(j)$ and assume that
		\begin{equation} \label{cond: DoubleSum}
		\sum_{i=1}^{\infty}\bigl(\sum_{j=1}^{\infty}\tilde{Q}(ij)\bigr)^{\frac{d+1}{2}}<\infty. 
		\end{equation}
		If $(\Vert Q \Vert_{\frac{d+1}{2}, \mathbb{Z}}, \Vert Q \Vert_{d+1, \mathbb{Z}\setminus \{0\}}) \in G_d^{o}$ (the interior of the good set) then there exists a $q_0(Q,d)$ such that for all $q\geq q_0$ there exist 
		tree-automorphism invariant GGMs coming from $q$-periodic boundary law solutions 
		which are not equal to the free state. 
	\end{enumerate}
\end{thm}  
\begin{rk} \label{rk: ConcSum}
	For the SOS model $Q(i):=\exp(-\beta \vert i \vert)$, condition \ref{cond: DoubleSum} is satisfied at any $\beta>0$. In the case of the logarithmic potential $Q(i):=\frac{1}{(1+ \vert i \vert)^\beta}$ we have $Q \notin l_1(\Z)$ if $\beta \leq 1$. On the other hand, for $\beta>1$ even condition \ref{cond: DoubleSum} is satisfied.
\end{rk}
\subsection{Localization vs. delocalization}
Localized Gibbs measures and delocalized gradient Gibbs measures can be distinguished by samples along paths, as the following theorem states.
\begin{thm} \label{thm: deloc}
	If $\lambda_q$ is a $q$-periodic boundary law solution for $Q$ then the gradient Gibbs measure $\nu_q$ associated to it via \eqref{Coexeq: ConstructionOfGGM} is different from the projection of the localized Gibbs measures given by Theorem \ref{Coexthm: main}.
	More precisely, the former one \textit{delocalizes} in the sense that $\nu^{\lambda_q}(W_n=k) \stackrel{n \rightarrow \infty}{\rightarrow} 0$ for any total increment $W_n$ along a path of length $n$ and any $k \in \Z$. 
	
	On the other hand, let $\nu$ be the projection to the gradient space of any of the Gibbs measures whose existence is guaranteed by Theorem \ref{Coexthm: main}. Then $\nu$ is localized in the sense that for all $k \in \Z$
	the probability $\nu(W_n=k)$ has a strictly positive limit as $n$ tends to infinity (see \eqref{eq: Localization}).
\end{thm}	
This shows that both types of measures behave fundamentally different. 
\subsection{Identifiability via boundary laws}
Do different boundary laws really define different gradient measures?
The following theorem positively answers this question.

\begin{thm} \label{thm: ident}
	Let $Q$ be any symmetric transfer operator for some gradient interaction potential and let $\lambda_{q_1}$ and $\lambda_{q_2}$ be two spatially homogeneous height-periodic boundary laws for $Q$ with minimal periods $q_1$ and $q_2$, respectively.
	Then the following holds true for the associated gradient Gibbs measures:
	
	If $\nu^{\lambda_{q_1}}=\nu^{\lambda_{q_2}}$ then $q_1=q_2$ and there are some cyclic permutation $\rho \in S_{q_1}$ and some constant $c>0$ such that $\lambda_{q_2}=c \, \lambda_{q_1} \circ \rho$.
\end{thm}
\begin{cor} \label{cor: Period} 
	Let $q \geq 2$ and $\lambda_q$ be some $q$-height-periodic boundary law for a transfer operator $Q$.		
	Let $\Delta^q \subset \mathbb{R}^q$ denote the $q-1$-dimensional standard simplex and consider $u,v \in \Delta^q$ as equivalent iff $u=v \circ \rho$ for some cyclic permutation of coordinates $\rho$. 
	Consider any infinite path of edges $b_1,b_2,\ldots$ and the height-field along this path defined by prescribing  a fixed height at some vertex of the path. Then for any $\tilde{q} \in \{2,3, \ldots\}$ the empirical distribution of the mod-$\tilde{q}$ projections of heights along the path almost surely converges to a deterministic limit $(\alpha_{\tilde{q}}(0), \ldots, \alpha_{\tilde{q}}(\tilde{q}-1))  \in \Delta^{\tilde{q}}/\sim$.   	
	The minimal period $q$ of the underlying boundary law can be recovered as the greatest common divisor of all $\tilde{q} \in \{2,3, \ldots\}$ with the property that the associated tentative boundary law $\tilde{\lambda}_{\tilde{q}}=(\alpha_{\tilde{q}}(0)^\frac{d}{d+1}, \ldots, \alpha_{\tilde{q}}(\tilde{q}-1)^\frac{d}{d+1})$ is indeed a $\tilde{q}$-periodic boundary law for $Q$ and $\nu^{\tilde{\lambda}_{\tilde{q}}}=\nu^{\lambda_{q}}$.
\end{cor} 

\section{Applications} \label{Coexsec: Applications}
Theorems \ref{Coexthm: main} and \ref{thm: main2} state existence of (gradient) Gibbs measures if a pair of certain $p$-norms of the transfer operator $Q$ lies in the so called good set $G_d$ \eqref{eq: good set}. To understand this good set better, we will look first at the extreme cases of the binary tree and trees of 
large degrees, still  for general potentials. 
Then, we will treat in more detail the 
SOS model (with exponentially fast decay of $Q$) 
and the log-potential (with polynomially slow decay of $Q$)
on general trees, where we discuss coexistence and non-coexistence of localized Gibbs measures 
and delocalized gradient Gibbs measures.  

\subsection{Binary tree}
In the case of the binary tree, the good set can be explicitly described by the hypograph of a function pointwise given by a root of a polynomial equation of order four.
\begin{pro} \label{pro: GoodBinary}
	Consider the binary tree. Then the good set $G_2 \subset (1, \infty) \times (0, \infty)$ is bounded by the graph of the function $\delta: (1, \infty) \rightarrow (0, \infty)$ defined by 
	\begin{equation*}
	\begin{split}
	\delta(\gamma) &:=
	\frac{1}{2}\sqrt{2\frac{\gamma^3-\frac{1}{4}}{\sqrt{(\gamma^3+\frac{1}{4})^\frac{2}{3}-\gamma}}-(\gamma^3+\frac{1}{4})^\frac{2}{3}-2\gamma}-\frac{1}{2}\sqrt{(\gamma^3+\frac{1}{4})^\frac{2}{3}-\gamma} \cr
	& = \frac{3}{16}\gamma^{-1}+O(\gamma^{-4}).
	\end{split}
	\end{equation*}
\end{pro}
\begin{rk}
	$\delta(\gamma)$ is the unique positive root to the equation
	\begin{equation*}
	16\gamma^2\delta^4+24\gamma^3\delta^2+(16\gamma^5-4\gamma^2)\delta-3\gamma^4 =0.
	\end{equation*} 
\end{rk}
\begin{center}
	\begin{figure}[h]
		\centering
		\includegraphics[width=10cm]{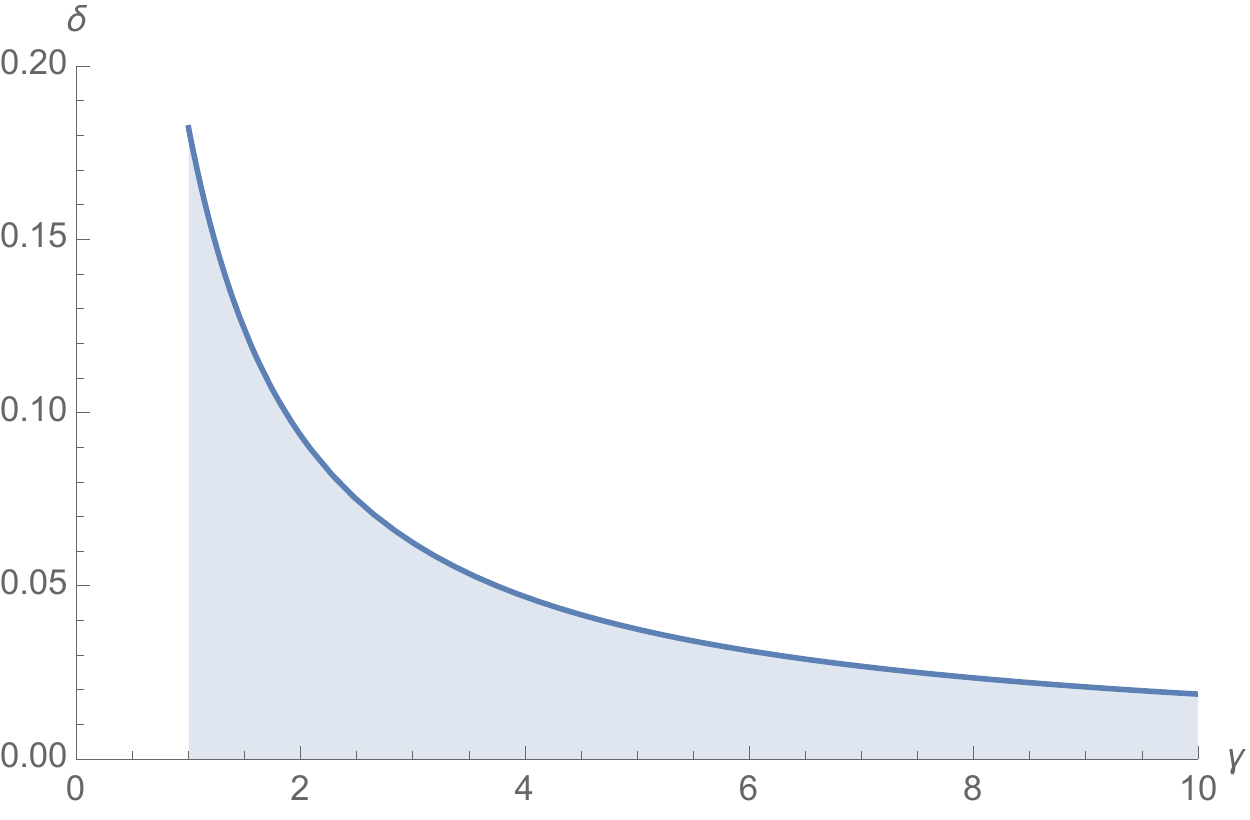}
		\caption{The good set $G_2$ embedded in the upper right quadrant of $\mathbb{R}^2$. }
		\label{fig: GoodSet}
	\end{figure}
\end{center}
\newpage 
\subsection{Large degree asymptotics}
We have the following model-independent result for large degrees $d$.
\begin{thm} \label{thm: LargeDegree}
	Let $U$ be some symmetric gradient interaction potential for a $\Z^k$-valued random field on the $d§$-regular tree with $U(0)=0$ and	
	$v:=\inf_{j \in \Z^k \setminus \{0\}}U(j)>0$. 
	Let $A>\frac{1}{v}$ be any fixed number and set
	\begin{equation*}
	\beta_{A,d}:=A \log{d}/(d+1).
	\end{equation*}
	Assume that  for the associated transfer operator $Q_\beta:=\exp(-\beta U)$ we have \\$\Vert Q_{\beta_{A,d}} \Vert_{\frac{d+1}{2},\Z^k}<\infty$ for some $d \geq 2$.
	Then the following holds true.
	\begin{enumerate}
		\item  There is a minimal degree $d_0 \geq 2$ such that for all $d \geq d_0$ there is a family of distinct tree-automorphism invariant Gibbs measures $(\mu_i)_{i \in \Z^k}$ for the transfer operator $Q_\beta$ at any $\beta \geq \beta_{A,d}$. 
		\item In this range the measures $\mu_i$ satisfy the following concentration bounds.
		\[\frac{\mu_i(\sigma_0 \neq i)}{\mu_i(\sigma_0=i)}\leq C \frac{1}{d} \stackrel{d \rightarrow \infty}{\rightarrow}0,\]
		where $C>0$ is some constant. 
	\end{enumerate}
\end{thm}	
The analogous large-degree existence results for the gradient Gibbs measures with local state space $\Z$ can be derived under summability of $Q$ and condition \ref{cond: DoubleSum}.
\subsection{Examples: SOS model and log-potential} \label{sec: Examples}
We illustrate the theory developed above by two concrete examples 
with local state space $\Z$ for a range of finite degrees. In both cases, the transfer operator (the potential respectively) is parametrized by the inverse temperature $\beta>0$. Hence the respective parameters $\gamma$ and $\delta$ in Theorem \ref{Coexthm: main} are both functions of $\beta>0$ whose values are obtained by carrying out the corresponding series.
\begin{figure}[h]
	\centering
	\begin{tabular}{l||l|l|l}
		Model & $Q(i)$& $\gamma_d(\beta)$ & $\delta_d(\beta)$ \\
		\hline 
		SOS &$\exp(-\beta \vert i \vert)$&$\tanh(\frac{d+1}{4}\beta)^{-\frac{2}{d+1}}$&$(\frac{2}{\exp((d+1)\beta)-1})^{\frac{1}{d+1}}$  \\
		Log-potential &$\frac{1}{(1+\vert i \vert)^\beta}$&$\big(2\zeta(\frac{d+1}{2}\beta)-1 \big)^\frac{2}{d+1}$&$\big(2(\zeta( \, (d+1)\beta \, )-1) \big)^\frac{1}{d+1}$	
	\end{tabular}
	\caption{The two models and their respective parameters. Here $\zeta(s)=\sum_{i=1}^{\infty}(\frac{1}{i})^s$ denotes the Riemann zeta-function.}
\end{figure}

Inserting the functions $\gamma_d$ and $\delta_d$ into \eqref{eq: Ball} and \eqref{eq: L-const}, i.e., calculating both the size of the minimal invariant ball and the value of the respective Lipschitz-constant as a function of $\beta$, numerical calculation with MATHEMATICA gives the following Figure \ref{fig: ConcreteExamples} showing infima of inverse temperatures on which our method ensures the existence of an invariant ball with Lipschitz constant smaller than one. 
\begin{figure}[h]
	\centering
	\begin{tabular}{l||l|l}
		$d$ & $\beta_{d,\text{SOS}}$  & $\beta_{d,\text{Log}}$ \\
		\hline
		$2$&1.997 &2.908 \\
		$3$&1.321 &1.930\\
		$6$&0.7240&1.057\\
		$7$& 0.6198&0.9297 \\
		$100$&0.06946&0.1005 \\
		$1000$&$9.238*10^{-3}$ & 0.01334  \\
		$10^{10}$&$2.536*10^{-9}$&$3.658*10^{-9}$
	\end{tabular}
	\caption{Infima of inverse temperatures for which the pair of parameters \\$\gamma_d(\beta)=\Vert Q \Vert_{\frac{d+1}{2}, \Z}$ and $\delta_d(\beta)=\Vert Q \Vert_{d+1,\Z \setminus \{0\}}$ lies in the good set $G_d$. In this case, Theorem \ref{Coexthm: main} guarantees the existence of a family of tree-automorphism invariant Gibbs measures. In view of the second statement of Theorem \ref{thm: main2} and Remark \ref{rk: ConcSum}, the following holds true for delocalized gradient Gibbs measures. For the SOS model above these thresholds also a countable family of delocalized gradient Gibbs measures exist. For the logarithmic potential this is true if and only if $d \leq 6$, as delocalized gradient Gibbs measures can not exist at inverse temperatures below $1$. All numbers are given with four-digit precision.}
	\label{fig: ConcreteExamples}
\end{figure}

Let $d\geq 2$ and $q \geq 2$. Theorem \ref{thm: main2} on the existence of (delocalized) gradient Gibbs measures and Theorem \ref{Coexthm: main} on the existence of localized Gibbs measures were both formulated in terms of the same good set $G_d \subset (1,\infty) \times (0,\infty)$ but for different norms. More precisely, the proof of Theorem \ref{thm: main2} is based on showing that on condition of either the vector of the $1$-norms $(\Vert Q \Vert_{1, \Z}, \Vert Q \Vert_{1, \Z \setminus \{0\}})$ or the pair $(\gamma,\delta)=(\Vert Q\Vert_{\frac{d+1}{2}, \Z}, \Vert Q \Vert_{d+1, \Z \setminus \{0\}})$ considered in Theorem \ref{Coexthm: main} lying in the good set, the vector \[(\gamma_q,\delta_q)=(\Vert Q_q\Vert_{\frac{d+1}{2}, \Z_q}, \Vert Q_q \Vert_{d+1, \Z_q \setminus \{\bar{0}\}})\] of respective norms of the \textit{fuzzy} transfer operator lies in the good set $G_d$ either for all $q \geq 2$, or at least for all $q$ sufficiently large, respectively. This then ensures existence of $q$-periodic delocalized GGMs for the respective values of $q$. 

Let us illustrate these different norms in the concrete examples of the SOS model and the log-potential. In Figure \ref{Fig: GoodSetCurves} below we give two numerically computed pictures of the good set for the Cayley tree of order $d=3$ 
supplemented 
with a few of such curves parametrized by $\beta$.
The fact that
\[(\gamma_q,\delta_q) \stackrel{q \rightarrow \infty}{\rightarrow} (\gamma,\delta)\]
motivates our notation  $q=\infty$ for the vector $(\gamma,\delta)$ in Figure \ref{Fig: GoodSetCurves}.
For the explicit computation of the finite-$q$ norms in the more complicated case of the log-potential we used that the fuzzy transfer operators $Q_q$ can be expressed in terms of the \textit{Hurwitz zeta function}
$\zeta(s,a)=\sum_{n=0}^{\infty}\frac{1}{(n+a)^s}$ which is defined for  $\mathcal{R}(s)>1$ and $ 0<\mathcal{R}(a)\leq 1.$ 
\begin{center}
	\begin{figure}[]
		\centering
		\subfloat[SOS model with $\beta \in \lbrack 1.2,3 \rbrack$ ]
		{\includegraphics[width=10cm]{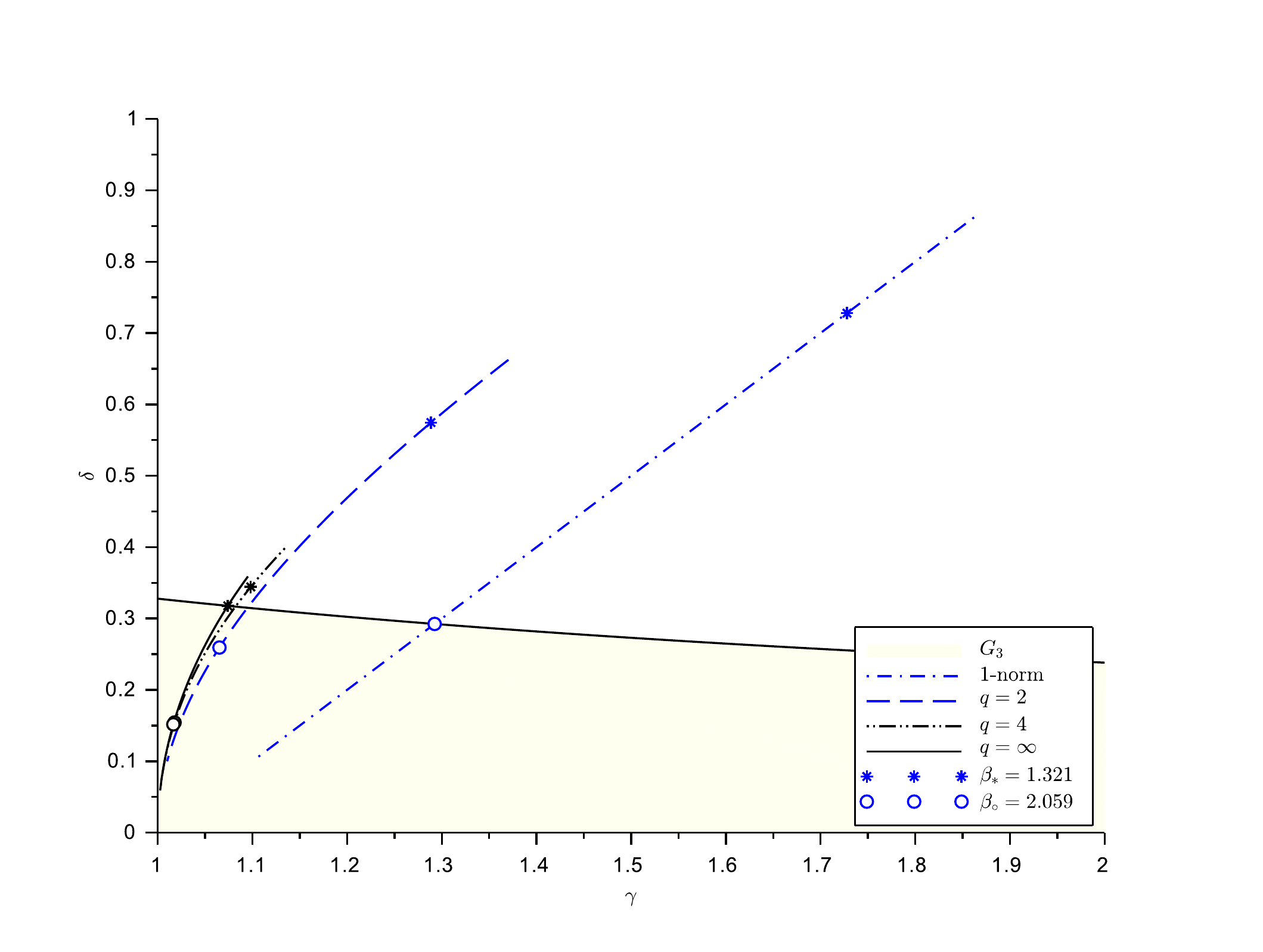}}
		\quad
		\subfloat[Model with log-potential and $\beta \in \lbrack 1.9,5 \rbrack$ ]{\includegraphics[width=10cm]{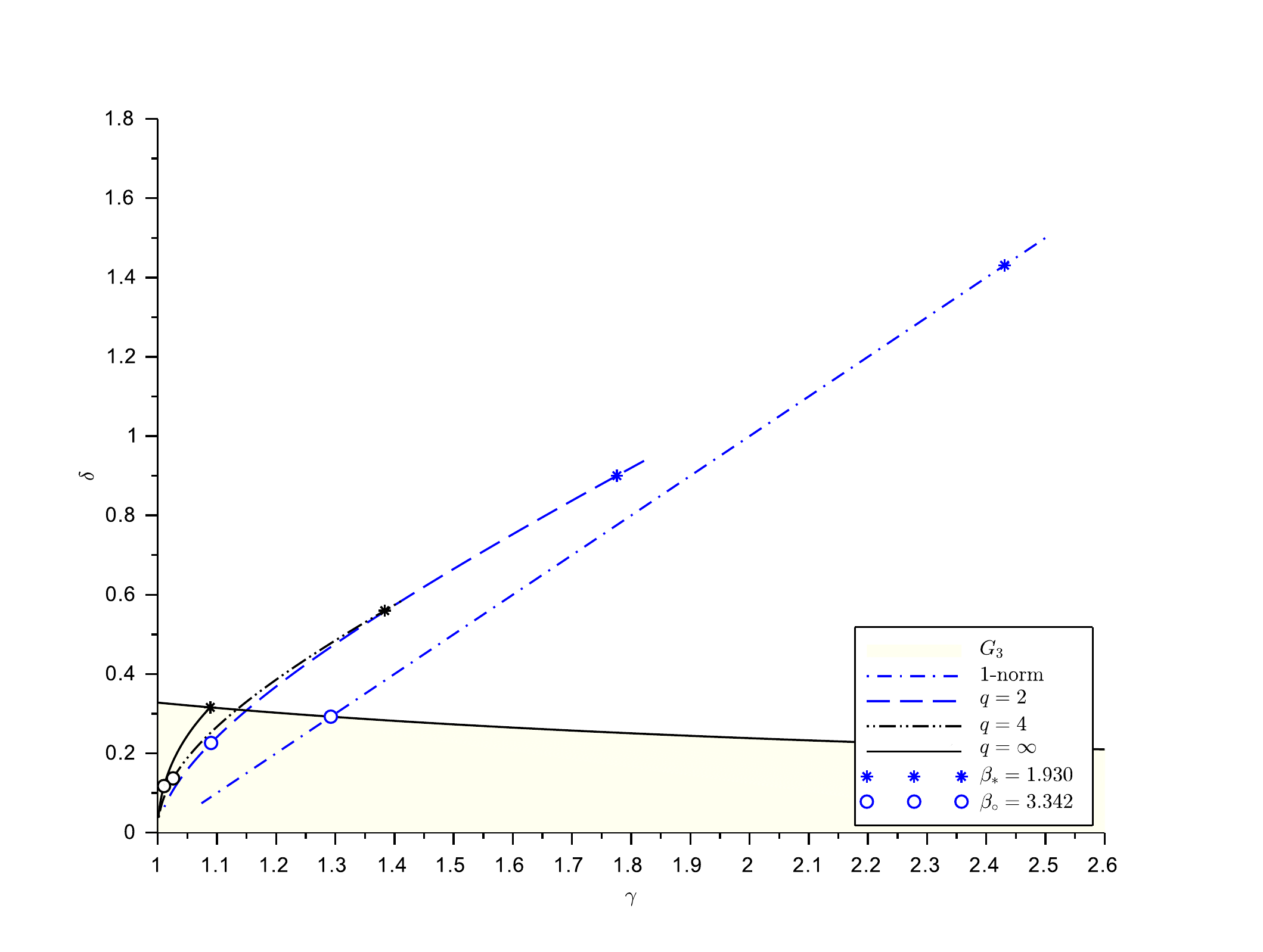}}
		\caption{  The good set $G_3$ for the Cayley tree of order $d=3$ bounded from above by the numerically computed solid line. It is supplemented with the curves $(\gamma_q,\delta_q)$ parametrized by $\beta$ within the given ranges and the curve $(\Vert Q \Vert_{1, \Z}, \Vert Q \Vert_{1, \Z \setminus \{0\}})$ of the $1$-norms. 
			The curves start from top right and enter the good set upon increase of $\beta$. 		
			The asterisks mark the values at $\beta_*$ on the different curves for which the $(\gamma,\delta)$-curve enters the good set. The circles mark the situation at $\beta_\circ$ where the curve of the $1$-norm enters the good set. The numbers are given with four-digit precision.
		}
		\label{Fig: GoodSetCurves}
	\end{figure}
\end{center}
\newpage

\section{Proofs} \label{Coexsec: Proofs}
The proof of Lemma \ref{lem: existenceModel} will be postponed to the end of this section.
\begin{proof}[Proof of Lemma \ref{lem: ineq}]
	Noticing that the denominator of $T(x)$ is bounded from below we have
	\begin{equation*}
	\begin{split}
	\Vert T(x) \Vert_{d+1, \mathbb{Z}^k \setminus \{0\}} & \leq \Vert Q(\cdot)+\sum_{j \in \mathbb{Z}^k \setminus \{0\}} Q(\cdot-j) \vert x(j) \vert^d \Vert_{d+1, \mathbb{Z}^k \setminus \{0\}}  \cr& \leq \Vert Q \Vert_{d+1, \mathbb{Z}^k \setminus \{0\}}+\Vert \sum_{j \in \mathbb{Z}^k \setminus \{0\}} Q( \cdot-j)\vert x(j)\vert^d \Vert_{d+1, \mathbb{Z}^k \setminus \{0\}}.     \cr
	\end{split}
	\end{equation*}
	Let $\tilde{x}(i):=\begin{cases}
	x(i), \quad &\text{if } i \neq 0 \\
	0, \quad &\text{if } i =0
	\end{cases}$ denote the extension of $x$ to $\Z^k$ by $0$. 
	Then the second term can be estimated from above by a convolution in $l_{d+1}(\mathbb{Z}^k)$
	\begin{equation} \label{eq: canonicalEmbedding} 
	\begin{split}
	\Vert \sum_{j \in \mathbb{Z}^k \setminus \{0\}} Q( \cdot-j)\vert x(j)\vert^d \Vert_{d+1, \mathbb{Z}^k \setminus \{0\}}& \leq	\Vert \sum_{j \in \mathbb{Z}^k} Q( \cdot-j)\vert \tilde{x}(j)\vert^d \Vert_{d+1, \mathbb{Z}^k}       \cr
	& =: \Vert Q*\vert \tilde{x}(j) \vert^d \Vert_{d+1, \mathbb{Z}^k}. 
	\end{split}
	\end{equation} 
	Now we want to apply Young's inequality for convolutions of Borel-measurable functions on unimodal locally compact groups with respect to the Haar-measure (cf. Theorem 20.18 in \cite{He79}) to $l_p(\mathbb{Z}^k)$ \begin{equation} \label{Young}\Vert u * v \Vert_{r, \mathbb{Z}^k} \leq \Vert u \Vert_{p, \mathbb{Z}^k} \, \Vert v \Vert_{q, \mathbb{Z}^k}, \, 1+\frac{1}{r}=\frac{1}{p}+\frac{1}{q} \end{equation} where $u(i):= Q(i)$, $v(j)
	:= \vert \tilde{x}(j) \vert^d$ and $q=\frac{d+1}{d}$, $r= d+1$, $p=\frac{d+1}{2}$.
	
	This gives
	\begin{equation*}
	\begin{split}
	\Vert T(x) \Vert_{d+1, \mathbb{Z}^k \setminus \{0\}}& \leq \Vert Q \Vert_{d+1, \mathbb{Z}^k \setminus \{0\}}+ \Vert Q \Vert_{\frac{d+1}{2}, \mathbb{Z}^k} \, \Vert \tilde{x} \Vert^d_{\frac{d+1}{d},\mathbb{Z}^k} \cr
	&=\Vert Q \Vert_{d+1, \mathbb{Z}^k \setminus \{0\}}+\Vert Q \Vert_{\frac{d+1}{2}, \mathbb{Z}^k} \, \Vert x \Vert_{d+1,\mathbb{Z}^k \setminus \{0\}}^d. 
	\end{split}
	\end{equation*}
	By our assumptions the r.h.s. now is a finite number, concluding the proof of the Lemma.
\end{proof}
The proof of Proposition \ref{thm: cond. Banach} is based on the following Lemma:
\begin{lemma} \label{lem: Lipschitz}
	For all $x,y \in D$ and any $i \in \mathbb{Z}^k \setminus \{0\}$ the following holds.
	\begin{equation*}
	\begin{split}
	&\vert T(x)(i)-T(y)(i) \vert \cr 
	&\leq \sum_{j \in \mathbb{Z}^k \setminus \{0\}} Q(i-j) \vert x(j)^d-y(j)^d \vert \cr & \quad + \left(\sum_{k \in \mathbb{Z}^k \setminus \{0\}}Q(k) \vert x(k)^d-y(k)^d \vert \right) \left(Q(i)+ \sum_{j \in \mathbb{Z}^k \setminus \{0\}}Q(i-j)  y(j)^d  \right)
	\end{split}
	\end{equation*}
\end{lemma}
\begin{proof}[Proof of Lemma \ref{lem: Lipschitz}]
	Write $T(x)(i)=: \frac{Z(x)}{N(x)}$, i.e., $Z(x)$ and $N(x)$ are the nominator (denominator resp.) of $T(x)$ as defined in \eqref{Def: Operator}. Then we have:
	\begin{equation*}
	\begin{split}
	\vert T(x)(i)-T(y)(i) \vert 
	& = \vert \frac{Z(x)}{N(x)}-\frac{Z(y)}{N(y)} \vert \leq \frac{\vert Z(x)-Z(y)\vert }{N(x)}+Z(y)\frac{\vert N(y)-N(x) \vert}{N(x)N(y)} \cr
	& \leq \vert Z(x)-Z(y) \vert +Z(y) \vert N(y)-N(x) \vert, 
	\end{split}
	\end{equation*} 
	where the second inequality follows from the fact that $N(x)$ and $N(y)$ are bounded from below by $1$. This completes the proof of the Lemma.
\end{proof}
\begin{proof}[Proof of Proposition \ref{thm: cond. Banach}] 
	\mbox{} \\
	\begin{enumerate}
		\item The first statement of the Proposition follows immediately from Lemma \ref{lem: ineq}.
		\item  To prove the second statement of the Proposition, i.e., Lipschitz-continuity, consider any $x,y \in B_\varepsilon(0) \cap D$. By Lemma \ref{lem: Lipschitz} and the triangle inequality we have
		\begin{equation} \label{eq: Lipschitz}
		\begin{split}
		&\Vert T(x)-T(y) \Vert_{d+1, \mathbb{Z}^k \setminus \{0\}} \cr 
		&\leq \Vert \sum_{j \in \mathbb{Z}^k \setminus \{0\}} Q( \cdot-j) \vert x(j)^d-y(j)^d \vert \,  \Vert_{d+1, \mathbb{Z}^k \setminus \{0\}} \cr & \quad +  \left(\sum_{k \in \mathbb{Z}^k \setminus \{0\}}Q(k) \vert x(k)^d-y(k)^d \vert \right) \cr 
		& \qquad \, \times \Vert Q(\cdot)+\sum_{j \in \mathbb{Z}^k \setminus \{0\}}Q( \cdot -j)  y(j)^d \Vert_{d+1, \mathbb{Z}^k \setminus \{0\}}.
		\end{split}
		\end{equation} 
		We start with estimating the second term. First note that $y \in B_\varepsilon(0) \cap D$ implies
		\begin{equation}
		\Vert Q(\cdot) +\sum_{j \in \mathbb{Z}^k \setminus \{0\}}Q( \cdot -j)  y(j)^d \Vert_{d+1, \mathbb{Z}^k \setminus \{0\}}\leq \varepsilon. 
		\end{equation}
		In what follows, we employ the fact that for any real numbers $a,b$ we have
		\begin{equation*}
		a^d- b^d=(a-b)(a^{d-1}+ba^{d-2}+ \ldots +b^{d-2}a+b^{d-1}),
		\end{equation*}
		so for any $j \in \mathbb{Z}^k \setminus \{0\}$, 
		\begin{equation} \label{KeyEst}
		\begin{split}
		\vert x(j)^d-y(j)^d \vert &\leq \vert x(j)-y(j) \vert d \max(\vert x(j) \vert, \vert y(j) \vert)^{d-1} \cr 
		& \leq d \,  \vert x(j)-y(j) \vert \, (\vert x(j) \vert^{d-1}+\vert y(j) \vert^{d-1}).
		\end{split}
		\end{equation}
		For the prefactor of the second term on the r.h.s. of \eqref{eq: Lipschitz} we thus obtain 
		\begin{equation*}
		\begin{split}
		&\sum_{k \in \mathbb{Z}^k \setminus \{0\}}Q(k) \vert x(k)^d-y(k)^d \vert \cr 
		& \leq d\sum_{k \in \mathbb{Z}^k \setminus \{0\}}Q(k) \vert x(k)-y(k) \vert x(k)^{d-1} +d\sum_{k \in \mathbb{Z}^k \setminus \{0\}}Q(k) \vert x(k)-y(k) \vert y(k)^{d-1}  \cr 
		&=d\Vert Q \, \vert x-y \vert \, x^{d-1}\Vert_{1, \mathbb{Z}^k \setminus \{0\}}+d\Vert Q \, \vert x-y \vert \,  y^{d-1}\Vert_{1, \mathbb{Z}^k \setminus \{0\}} \cr 
		& \leq d\Vert Q \Vert_{d+1, \mathbb{Z}^k \setminus \{0\}} \ \Vert x-y \Vert_{d+1, \mathbb{Z}^k \setminus \{0\}} \, (\Vert x\Vert_{d+1, \mathbb{Z}^k \setminus \{0\}}^{d-1}+\Vert y\Vert_{d+1, \mathbb{Z}^k \setminus \{0\}}^{d-1}) \cr 
		& \leq 2d \delta \varepsilon^{d-1} \Vert x-y\Vert_{d+1, \mathbb{Z}^k \setminus \{0\}},  
		\end{split}
		\end{equation*} 
		where the second inequality follows from applying a generalized version of H\"older's inequality 
		\begin{equation*}
		\Vert u v w\Vert_{1, \mathbb{Z}^k \setminus \{0\}} \leq \Vert u\Vert_{p, \mathbb{Z}^k \setminus \{0\}}^\frac{1}{p} \, \Vert v\Vert_{q, \mathbb{Z}^k \setminus \{0\}}^\frac{1}{q} \, \Vert w\Vert_{r, \mathbb{Z}^k \setminus \{0\}}^\frac{1}{r}, \quad 1=\frac{1}{p}+\frac{1}{q}+\frac{1}{r} 
		\end{equation*} 
		to $u=Q$, $v= \vert x-y \vert$ and $w=x^{d-1}$ ($w=y^{d-1}$, respectively) with $p=q=d+1$ and $r=\frac{d+1}{d-1}$.
		It remains to estimate the first term on the r.h.s. of \eqref{eq: Lipschitz}.
		Similar to \eqref{eq: canonicalEmbedding}, 
		Young's inequality with $r=d+1$, $p=\frac{d+1}{2}$ and $q=\frac{d+1}{d}$ first gives
		\begin{equation*}
		\begin{split}
		\Vert \sum_{j \in \mathbb{Z}^k \setminus \{0\}} Q( \cdot-j) \vert x(j)^d-y(j)^d \Vert_{d+1, \mathbb{Z}^k \setminus \{0\}} 
		&\leq \Vert \sum_{j \in \mathbb{Z}^k} Q( \cdot-j) \vert \tilde{x}(j)^d-\tilde{y}(j)^d \vert \, \Vert_{d+1, \mathbb{Z}^k} \cr
		&\leq \Vert Q \Vert_{\frac{d+1}{2}, \mathbb{Z}^k} \Vert \tilde{x}^d-\tilde{y}^d \Vert_{\frac{d+1}{d}, \mathbb{Z}^k} \cr 
		&=\gamma \Vert x^d-y^d \Vert_{\frac{d+1}{d}, \mathbb{Z}^k \setminus \{0\}}.
		\end{split}
		\end{equation*}
		At this point we want to apply \eqref{KeyEst} in combination with H\"older's inequality in the form
		\begin{equation}
		\Vert uv \Vert_q \leq \Vert u \Vert_{q_1} \, \Vert v \Vert_{q_2} \text{where } 0<q,q_1,q_2< \infty \text{ and } \frac{1}{q}=\frac{1}{q_1}+\frac{1}{q_2} 
		\end{equation}  
		with $q=\frac{d+1}{d}$, $q_1=d+1$ and $q_2=\frac{d+1}{d-1}$ to obtain
		\begin{equation*}
		\begin{split}
		&\Vert x^d-y^d \Vert_{\frac{d+1}{d}, \mathbb{Z}^k \setminus \{0\}} \cr 
		& \leq d \Vert \, \vert x-y \vert \, (x^{d-1}+y^{d-1}) \Vert_{\frac{d+1}{d}, \mathbb{Z}^k \setminus \{0\}} \cr  
		& \leq d \Vert x-y \Vert_{d+1, \mathbb{Z}^k \setminus \{0\}}(\Vert x^{d-1} \Vert_{\frac{d+1}{d-1}, \mathbb{Z}^k \setminus \{0\}} + \Vert y^{d-1} \Vert_{\frac{d+1}{d-1}, \mathbb{Z}^k \setminus \{0\}}) \cr 
		&=d \Vert x-y \Vert_{d+1, \mathbb{Z}^k \setminus \{0\}}(\Vert x \Vert_{d+1, \mathbb{Z}^k \setminus \{0\}}^{d-1} + \Vert y \Vert_{d+1, \mathbb{Z}^k \setminus \{0\}}^{d-1} ) \cr 
		& \leq 2d \varepsilon^{d-1} \Vert x-y \Vert_{d+1, \mathbb{Z}^k \setminus \{0\}}.
		\end{split}
		\end{equation*}
		Inserting these estimates into \eqref{eq: Lipschitz}, we arrive at
		\begin{equation}
		\begin{split}
		&\Vert T(x)-T(y) \Vert_{d+1, \mathbb{Z}^k \setminus \{0\}} \cr 
		& \leq 2d\gamma \varepsilon^{d-1}\Vert x-y \Vert_{d+1, \mathbb{Z}^k \setminus \{0\}}+\varepsilon (2d \delta \varepsilon^{d-1} \Vert x-y\Vert_{d+1, \mathbb{Z}^k \setminus \{0\}}) \cr 
		&= \Vert x-y \Vert_{d+1, \mathbb{Z}^k \setminus \{0\}}2d(\gamma\varepsilon^{d-1}+\varepsilon^{d}\delta).
		\end{split}
		\end{equation}
		Hence $T$ is Lipschitz-continuous on $B_{\varepsilon}(0) \cap D$ with constant
		\begin{equation*}
		\begin{split}
		L&=2d(\gamma\varepsilon^{d-1}+\delta\varepsilon^{d}),
		\end{split}
		\end{equation*}
		which proves the second statement of the theorem.
	\end{enumerate}    
\end{proof}
\begin{proof}[Proof of Theorem \ref{Coexthm: main} ]
	Assume that $d \geq 2$ and $(\gamma, \delta) \in G_d$. Then, by definition of the set $G_d$, there is some $\varepsilon>0$ such that the equations \ref{eq: Ball} and \ref{eq: L-const} are satisfied.
	From Proposition \ref{thm: cond. Banach} it follows that $T$ leaves the the $\varepsilon$-ball $B_\varepsilon(0) \subset l_{d+1}(\mathbb{Z}^k \setminus \{0\})$ invariant and that $T$ restricted on $B_\varepsilon(0)$ is a contraction mapping, hence Banach's fixed point theorem guarantees the existence of a (unique) fixed point $x \in B_\varepsilon(0)$. Going over to $\bar{x} \in l_{d+1}(\mathbb{Z}^k)$ where $\bar{x}(i):=\begin{cases} 1, &\quad i=0, \\ x(i), & \quad \text{else} \end{cases}$ and taking the $d$th power we finally obtain a spatially homogeneous boundary law solution $\lambda\in l_{\frac{d+1}{d}}(\mathbb{Z}^k)$ with $\lambda(0)=1$. By Theorem \ref{Coex: BLMC} this normalizable boundary law solution corresponds to a unique tree-automorphism invariant Gibbs measure $\mu_0=\mu^\lambda$ for the Gibbsian specification \eqref{CoexDef: Gibbs specification}.
	As the transfer operator $Q$ is obtained from a gradient interaction potential, for any $i \in \Z^k$ the function $\lambda_i$ on $\Z^k$ given by $\lambda_i(j):=\lambda(j-i)$ will also satisfy the boundary law equation \eqref{eq: homBLeq}. Hence we obtain a whole family $(\lambda_i)_{i \in \Z^k}$ of boundary laws.
	To show that they are distinct, we note that each $\lambda_i$ is symmetric and, by construction, an element of the $\varepsilon$-ball in $l_{\frac{d+1}{d}}(\Z^k)$ centered arround the element that is one at site $i$ and zero elsewhere. Since $\varepsilon<1$ we conclude that the family $(\lambda_i)_{i \in \Z^k}$ is pairwise distinct. This implies that the measures $\mu_i$, $i \in \Z^k$, each associated to the respective $\lambda_i$ are also distinct which concludes the first part of the proof.

	In the next step we prove the localization bounds. First note that by construction of the Gibbs measure $\mu_\lambda$ we have the following single-site marginal:
	\begin{equation}
	\mu_0(\sigma_0=i)=\frac{\lambda(i)^\frac{d+1}{d}}{1+\sum_{j\in \mathbb{Z}^k\setminus \{0\}}\lambda(j)^\frac{d+1}{d}}
	=\frac{\bar{x}(i)^{d+1}}{1+\sum_{j\in \mathbb{Z}^k\setminus \{0\}}x(j)^{d+1}},
	\end{equation}
	hence
	\begin{equation}
	\frac{\mu_0(\sigma_0=i)}{\mu_0(\sigma_0=0)}=\bar{x}(i)^{d+1}.
	\end{equation}
	From this it follows 
	\begin{equation} \label{eq: auxLocalization}
	\frac{\mu_0(\sigma_0 \neq 0)}{\mu_{0}(\sigma_0=0)}=\Vert x \Vert_{d+1, \mathbb{Z}^k \setminus \{0\}}^{d+1}.
	\end{equation}
	
	Now we want to approximate $\Vert x \Vert_{d+1, \mathbb{Z}^k \setminus \{0\}}^{d+1}$ by $\Vert Q \Vert_{d+1, \mathbb{Z}^k \setminus \{0\}}^{d+1}$.
	First by the the fixed-point property 
	\begin{equation*}
	\begin{split}
	x(i)-Q(i)&=T(x)(i)-Q(i) \cr &=- Q(i)\frac{\sum_{j \in \mathbb{Z}^k \setminus \{0\}}Q(j)  \, \vert x(j) \vert^d }
	{1+\sum_{j \in \mathbb{Z}^k \setminus \{0\}}Q(j)  \, \vert x(j) \vert^d }
	+\frac{\sum_{j \in \mathbb{Z}^k \setminus \{0\}}Q(i-j)  \, \vert x(j) \vert^d }{1+\sum_{j \in \mathbb{Z}^k \setminus \{0\}}Q(j)  \, \vert x(j) \vert^d }.
	\end{split}
	\end{equation*}
	Bounding the denominators from below by $1$ then gives
	\begin{equation} \label{eq: xQ-Distance}
	\begin{split}
	\Vert x-Q\Vert_{d+1, \mathbb{Z}^k \setminus \{0\}}&\leq  \Vert Q\Vert_{d+1,\mathbb{Z}^k \setminus \{0\}}
	\sum_{j \in \mathbb{Z}^k \setminus \{0\}}Q(j)  \, \vert x(j) \vert^d \cr
	& \quad +\Vert
	\sum_{j \in \mathbb{Z}^k \setminus \{0\}}Q(\cdot -j)  \, \vert x(j) \vert^d \Vert_{d+1,\mathbb{Z}^k \setminus \{0\}}.
	\end{split}
	\end{equation}
	Now, application of H\"older's inequality with $1=\frac{1}{d+1}+\frac{d}{d+1}$ to the first term leads to
	\begin{equation*}
	\Vert Q\Vert_{d+1,\mathbb{Z}^k \setminus \{0\}}
	\sum_{j \in \mathbb{Z}^k \setminus \{0\}}Q(j)  \, \vert x(j) \vert^d \leq \Vert Q\Vert_{d+1,\mathbb{Z}^k \setminus \{0\}}^2  \Vert x \Vert_{d+1,\mathbb{Z}^k \setminus \{0\}}^d.
	\end{equation*} 
	Applying the same estimate as in \eqref{eq: canonicalEmbedding} to the second term, we arrive at
	\begin{equation*}
	\begin{split}
	&\vert \, \Vert x \Vert_{d+1, \mathbb{Z}^k \setminus \{0\}}-\Vert Q \Vert_{d+1, \mathbb{Z}^k \setminus \{0\}} \, \vert  \cr
	& \leq \Vert x-Q\Vert_{d+1, \mathbb{Z}^k \setminus \{0\}} \cr 
	& \leq \Vert Q\Vert_{d+1,\mathbb{Z}^k \setminus \{0\}}^2  \Vert x \Vert_{d+1,\mathbb{Z}^k \setminus \{0\}}^d+\Vert Q \Vert_{\frac{d+1}{2}, \mathbb{Z}^k}\Vert x\Vert_{d+1, \mathbb{Z}^k \setminus \{0\}}^d \cr 
	&=\Vert x \Vert_{d+1,\mathbb{Z}^k \setminus \{0\}}^d(\delta^2+\gamma).
	\end{split}
	\end{equation*} 
	Dividing both sides of the inequality by the positive number $\delta$ and writing
	
	$A:=\frac{1}{\delta}\Vert x \Vert_{d+1,\mathbb{Z}^k \setminus \{0\}}$ we obtain
	\begin{equation*}
	\vert A-1 \vert \leq  \delta \Vert x \Vert_{d+1,\mathbb{Z}^k \setminus \{0\}}^d+\gamma \Vert x\Vert_{d+1, \mathbb{Z}^k \setminus \{0\}}^{d-1}A, 
	\end{equation*}
	so 
	\begin{equation*}
	\frac{1-\delta\Vert x \Vert_{d+1,\mathbb{Z}^k \setminus \{0\}}^d}{1+\gamma\Vert x \Vert_{d+1,\mathbb{Z}^k \setminus \{0\}}^{d-1}} \leq A \leq \frac{1+\delta\Vert x \Vert_{d+1,\mathbb{Z}^k \setminus \{0\}}^d}{1-\gamma\Vert x \Vert_{d+1,\mathbb{Z}^k \setminus \{0\}}^{d-1}}.
	\end{equation*}
	Recalling the definition of $A$ and equation \eqref{eq: auxLocalization} and taking into account that \\$x \in B_\varepsilon(0)$ we arrive at the second statement of the theorem
	\begin{equation*}
	\left(\delta \frac{1-\delta\varepsilon^d}{1+\gamma \varepsilon^{d-1}} \right)^{d+1}\leq \frac{\mu_0(\sigma_0 \neq 0)}{\mu_0(\sigma_0=0)} \leq \left(\delta \frac{1+\delta\varepsilon^d}{1-\gamma \varepsilon^{d-1}} \right)^{d+1}.
	\end{equation*}
	Note that the proof and hence the theorem stay true if $\Z^k$ is replaced by the ring $\Z_q$ as all steps involving H\"older's and Young's inequalities are also valid.
\end{proof}

\begin{proof}[Proof of Theorem  \ref{thm: main2}]
	We look at the appropriate $q$-periodic boundary law solutions. 
	These can be rephrased in terms of length-$q$ boundary law solutions 
	for the $q$-spin model with $Q_q(\bar{i}):=\sum_{j\in \bar{i}}Q(j)=\sum_{j \in \Z}Q(i+qj)$. Thus if \begin{equation} \label{cond: goodSetgradient}
	(\Vert {Q_q} \Vert_{\frac{d+1}{2}, \Z_q \setminus \{0\}},\Vert Q_q \Vert_{d+1, \Z_q \setminus \{0\}}) \in G_d
	\end{equation} then existence of tree-automorphism invariant GGMs coming from $q$-periodic boundary law solutions follows from the version of Theorem \ref{Coexthm: main} for the local state space $\Z_q$. More precisely define $\bar{Q}_q(\bar{i}):=\frac{Q_q(\bar{i})}{Q_q(\bar{0})}$, $\bar{i} \in \Z_q$, and note that \ref{cond: goodSetgradient} implies \\$(\Vert \bar{Q}_q \Vert_{\frac{d+1}{2}, \Z_q \setminus \{0\}},\Vert \bar{Q}_q \Vert_{d+1, \Z_q \setminus \{0\}}) \in G_d$, hence the operator $T$ given by \ref{Def: Operator} for $\bar{Q}_q$ has a fixed point which provides the desired solution.

	Now, for any $1 \leq p<\infty$
	\begin{equation*}
	\begin{split}
	\Vert Q_q \Vert_{p, \Z_q}&=(\sum_{i \in \Z_q}Q_q(i)^p)^\frac{1}{p}=\big(\sum_{i=0}^{q-1}(\sum_{j \in \Z}Q(i+qj))^p \big)^\frac{1}{p} \cr 
	& \leq \sum_{i=0}^{q-1}\sum_{j \in \Z}Q(i+qj)=\Vert Q \Vert_{1, \Z}<\infty
	\end{split}
	\end{equation*} 
	and similarly $\Vert Q_q \Vert_{p, \Z_q \setminus \{0\}} \leq \Vert Q \Vert_{1, \Z \setminus \{0\}}$.
	This already proves the first statement of the theorem.
	
	To prove the second part, let $p \geq \frac{d+1}{2}$ and consider 
	\begin{equation*}
	f_q: \Z \rightarrow [0, \infty) \, ; \, f_q(i)=\begin{cases}
	(\sum_{j \in \Z}Q(i+qj))^p &\,  \text{if } i \in \{-\lfloor \frac{q}{2} \rfloor, \ldots, 0, \ldots, \lfloor \frac{q}{2} \rfloor\} \\
	0, &\, \text{else. }
	\end{cases}
	\end{equation*}
	We have $\Vert Q_q \Vert_p^p=\sum_{i \in \Z}f_q(i)-Q_{\lfloor \frac{q}{2} \rfloor}(i) \chi(q \text{ is even and } i=\frac{q}{2})$. Moreover, by the assumption $Q \in l_1(\Z)$ the family $(f_q)_{q \in \{2,3, \ldots\}}$ of non-negative functions is pointwisely converging to the function $f(\cdot)=Q(\cdot)^p$ on $\Z$. \\Similarly, $Q_{\lfloor \frac{q}{2} \rfloor}(i) \chi(q \text{ is even and } i=\frac{q}{2}) \stackrel{q \rightarrow \infty}{\rightarrow} 0$ for any fixed $i \in \Z$.
	
	In the following we will construct an integrable majorant for the family $(f_q)_{q \in \{k,k+1, \ldots\}}$ to be able to apply dominated convergence.
	First, going over to $\tilde{Q}$, we set 
	\begin{equation*}
	\tilde{f}_q: \Z \rightarrow [0, \infty) \, ; \, \tilde{f}_q(i)=\begin{cases}
	(\sum_{j \in \Z}\tilde{Q}(i+qj))^p &\quad \text{if } i \in \{-\lfloor \frac{q}{2} \rfloor, \ldots, 0, \ldots, \lfloor \frac{q}{2} \rfloor\} \\
	0, &\quad \text{else. }
	\end{cases}
	\end{equation*}
	Now for $q \in \{2,3, \ldots \}$ define the function
	\begin{equation*}
	g_q: \Z \rightarrow [0, \infty) \, ; \, g_q(i)= \begin{cases}
	\tilde{f}_q(i), & \quad \text{if } \vert i \vert \leq \lfloor \frac{q}{2} \rfloor \\
	\tilde{f}_{2i}(i), &\quad \text{if } \vert i \vert>  \lfloor \frac{q}{2} \rfloor,
	\end{cases} 
	\end{equation*}
	which is supported on the whole integers $\Z$.
	Clearly $g_q \geq \tilde{f}_q \geq f_q$. As $\tilde{Q}(i)$ is by construction monotonically decreasing in $\vert i \vert$, we have
	\begin{equation*}
	g_{q+1}(i)-g_q(i)= \begin{cases}
	\tilde{f}_{q+1}(i)-\tilde{f}_q(i)\leq 0 &\quad \text{if } \vert i \vert \leq  \lfloor \frac{q}{2} \rfloor\\
	\tilde{f}_{q+1}(i)-\tilde{f}_{2i}(i) \leq 0 &\quad \text{if }  \lfloor \frac{q}{2} \rfloor < \vert i \vert \leq  \lfloor \frac{q+1}{2} \rfloor \\
	\tilde{f}_{2i}(i)-\tilde{f}_{2i}(i)=0 &\quad \text{if } \vert i \vert >  \lfloor \frac{q+1}{2} \rfloor.
	\end{cases}	
	\end{equation*}
	Thus the family $(g_q)_{q \in \{2,3, \ldots\}}$ is decreasing. 
	Hence we have $f_q \leq \tilde{f}_q \leq g_2$ for all \\$q \in \{2,3 \dots\}.$
	
	As $g_2(i)=(\sum_{j \in \Z}\tilde{Q}(i+2ij))^p$ for all $\vert i \vert \geq 1$, integrability of $g_2$ (more precisely of any element of the family $(g_q)_{q \in \{2,3, \ldots\}}$) is equivalent to finiteness of the 
	expression
	\begin{equation*}
	\sum_{i \in \Z \setminus \{0\}} (\sum_{j \in \Z} \tilde{Q}(i(1+2j)))^p =2^{1+p}\sum_{i=1}^{\infty}(\sum_{j=0}^{\infty}\tilde{Q}(i(1+2j)))^p.
	\end{equation*}
	From symmetry of $\tilde{Q}$ and monotonicity in $\vert i \vert$ it then follows that
	integrability of $g_2$ is also equivalent to
	\begin{equation*}
	\sum_{i=1}^{\infty}(\sum_{j=1}^{\infty}\tilde{Q}(ij))^p<\infty.
	\end{equation*}
	By assumption \ref{cond: DoubleSum} this holds true for $p=\frac{d+1}{2}$ and hence also for $p=d+1$.
	From dominated convergence it follows that
	\begin{equation*}
	\begin{split}
	\Vert Q_q \Vert_{p, \Z_q}^p&=\sum_{i \in \Z} f_q(i)-Q_{\lfloor \frac{q}{2} \rfloor}(i)\chi(q \text{ is even and } i =\frac{q}{2})  \stackrel{q \rightarrow \infty}{\rightarrow} \sum_{i \in \Z} Q(i)^p=\Vert Q \Vert_{p,\Z}^p.
	\end{split}
	\end{equation*} 
	Similarly we have $\Vert Q_q \Vert_{p,\Z_q \setminus \{0\}} \stackrel{q \rightarrow \infty}{\rightarrow} \Vert Q \Vert_{p, \Z \setminus \{0\}}$. The proof of the second statement is then concluded by the assumption $(\Vert Q \Vert_{\frac{d+1}{2}, \mathbb{Z}}, \Vert Q \Vert_{d+1, \mathbb{Z}\setminus \{0\}}) \in G_d^{o}$.
\end{proof}

\begin{proof}[Proof of Theorem \ref{thm: deloc}]
	\label{coexpr: deloc}
	Let $\nu$ denote the gradient Gibbs measure obtained by projecting any of the localized Gibbs measures given by Theorem \ref{Coexthm: main} to the space of gradient configurations. Further let $\nu^{\lambda_q}$ denote a fixed gradient Gibbs measure constructed from a $q$-periodic boundary law $\lambda_q$.
	We will show that the marginals on a fixed path $(b_1, b_2, \ldots, b_n)$ of length $n$ differ as $n$ becomes sufficiently large.  
	More precisely, let $\eta_b:=\sigma_x-\sigma_y$ denote the gradient spin variable along the edge $b=(x,y) \in \vec{L}$ and set $W_n:= \sum_{i=1}^n\eta_{b_i}.$ 
	Then we have
	\begin{equation*}
	\nu(W_n=k)=\sum_{i \in \Z}\alpha(i)P^n(i,i+k)
	\end{equation*}
	where $P$ is the transition operator for the irreducible aperiodic tree-indexed Markov chain (the localized Gibbs measure) associated to boundary law $\lambda$ given by \\$P(i,j)=\frac{Q( i-j ) \lambda(j)}{\sum_{l \in \Z}Q( i-l ) \lambda(l)}>0$ and $\alpha$ is its stationary distribution given by $\alpha(i)=\frac{\lambda(i)^\frac{d+1}{d}}{\Vert \lambda \Vert_{\frac{d+1}{d}}^{\frac{d+1}{d}}}$. Existence of the stationary distribution guarantees that the process is positive recurrent	
	(cp. Theorem 3.3.1 in \cite{Br99}) and hence ergodic. \\Thus $P^n(i,j) \stackrel{n \rightarrow \infty}{\rightarrow} \alpha(j)$ for any fixed $(i,j) \in \Z^2$ (see Theorem 4.2.1 in \cite{Br99}). Dominated convergence then gives
	\begin{equation} \label{eq: Localization}
	\nu(W_n=k)=\sum_{i \in \Z}\alpha(i)P^n(i,i+k) \stackrel{n \rightarrow \infty}{\rightarrow} \sum_{i \in \Z} \alpha(i)\alpha(i+k)>0 \text{ for any } k \in \Z.
	\end{equation}
	On the other hand we will show that $\nu^{\lambda_q}(W_n=k) \stackrel{n \rightarrow \infty}{\rightarrow} 0$ for any $k \in \Z$, i.e., \textit{delocalization}.
	Let $(b_1,b_2, \ldots, b_n)$ be again any fixed path of length $n$ and $\sigma'=(\sigma'_{x_i})_{i=1, \ldots, n+1}$ be the fuzzy chain on $\Z_q$ along this path with respect to $P'_q$. 
	We have 
	\begin{equation*}
	\nu^{\lambda_q}(W_n=k) \, = \, E_{\sigma'}[\nu^{\lambda_q}(W_n=k \mid \sigma')].
	\end{equation*}  
	Writing $L_n^{\sigma'}(\bar{j}):=\frac{1}{n}\vert \{k \in \{1, \ldots, n\} \mid \sigma'_{x_{k+1}}-\sigma'_{x_{k}}=\bar{j} \} \vert$ for the empirical distribution of increments of the fuzzy chain this then reads
	\begin{equation*}
	W_n= \sum_{\bar{j} \in \Z_q}\sum_{a=1}^{nL_n^{\sigma'}(\bar{j})}X_a^{\bar{j}}.
	\end{equation*}
	Here the variables $(X_a^{\bar{j}})_{\bar{j} \in \Z_q, \, a=1, \ldots, nL_n^{\sigma'}(\bar{j})}$ are conditionally on $\sigma'$ independent and,  for fixed $\bar{j} \in \Z_q$, also identically distributed with distribution $\rho^Q_q( \cdot \mid \bar{j})$ (cp. equation \eqref{Coexeq: RWPEtr})
	yet they are not necessarily integrable.
	
	We will express the distribution of $W_n$ via its characteristic function.
	By Fourier-inversion for distributions on $\Z$ (cp. Thm 15.10 in \cite{Kl14}) we have
	\begin{equation*}
	\nu^{\lambda_q}(W_n=k \mid \sigma')=\frac{1}{2\pi}\int_{-\pi}^{\pi}E[\exp(iW_n t) \mid \sigma']  \exp(-itk) \text{d}t
	\end{equation*}
	Now for any fixed $t$ conditional independence gives 
	\begin{equation*} 
	\begin{split}
	\vert E[\exp(iW_n t) \mid \sigma']  \exp(-itk)  \vert &= \vert  \prod_{\bar{j} \in \Z_q}\prod_{a=1}^{nL_n^{\sigma'}(\bar{j})}E[\exp(itX_a^{\bar{j}}) \mid \sigma'] \vert \cr
	&= \prod_{\bar{j} \in \Z_q} \vert E[\exp(itX_1^{\bar{j}}) \mid \sigma'] \vert^{nL_n^{\sigma'}(\bar{j})} \cr
	&\leq   (\max_{\bar{j} \in \Z_q} \vert E[\exp(itX_1^{\bar{j}}) \mid \sigma'] \vert)^{n \sum_{\bar{j} \in \Z_q}L_n^{\sigma'}(\bar{j})} \cr 
	&=  (\max_{\bar{j} \in \Z_q}\vert E[  \exp(itX_1^{\bar{j}}) \mid \sigma'] \vert)^n \cr
	&= (\max_{\bar{j} \in \Z_q} \vert E[\exp(itX_1^{\bar{j}})] \vert \chi(L_n^{\sigma'}(\bar{j})>0))^n \cr 
	&\leq (\max_{\bar{j} \in \Z_q}  \vert E[\exp(itX_1^{\bar{j}})] \vert)^n.
	\end{split}
	\end{equation*}
	Exercise 11 of Chapter 9.5 (p.314) in \cite{Dud02} says that if $f$ is the characteristic function of a law $P$ and if there are $s,t$ with $t \neq 0$ and $\frac{s}{t} \in \mathbb{R} \setminus \mathbb{Q}$ such that $\vert f(s) \vert = \vert f(t) \vert=1$ then $P$ must be a Dirac measure. In other words, if $P$ is not a Dirac measure then the set of points where the characteristic function achieves an absolute value equal to $1$ has Lebesgue measure $0$.
	As $Q$ is strictly positive, the above aspects combine to
	\begin{equation*}
	\vert E[\exp(iW_n t) \mid \sigma']  \exp(-itk) \vert \stackrel{n \rightarrow \infty}{\rightarrow} 0 \text{ for }\lambda- \text{a.a. } t. 
	\end{equation*}
	uniformly in $\sigma'$.
	Thus by dominated convergence
	\begin{equation*}
	\begin{split}
	\frac{1}{2\pi}\int_{-\pi}^{\pi} \vert E[\exp(iW_n t) \mid \sigma']  \exp(-itk) \vert \text{d}t \stackrel{n \rightarrow \infty}{\rightarrow} 0
	\end{split}
	\end{equation*}
	uniformly in $\sigma'$ which concludes the proof of the Theorem.
\end{proof}

\begin{proof}[Proof of Theorem \ref{thm: ident}]
	First assume that $q_1=q_2$. From $\nu^{\lambda_{q_1}}=\nu^{\lambda_{q_2}}$ it follows that the distributions of increments of the underlying fuzzy chain must be the same for both $\lambda_{q_1}$ and $\lambda_{q_2}$. Applying the Ergodic Theorem for Markov chains, the statement of the proof then follows by the fact that observing the increments of the fuzzy chain along an infinite path on the tree allows to identify the underlying boundary law up to permutational invariance and multiplication by positive constants. 
	
	More precisely, similar to the proof of Theorem \ref{thm: deloc}, we consider any path\\ $(b_1,b_2, \ldots, b_n)=((x_0,x_1), (x_1,x_2), \ldots, (x_{n-1},x_n))$ of length $n$ and let \\$\sigma'=(\sigma'_{x_i})_{i=1, \ldots, n+1}$ denote the fuzzy chain on $\Z_q$ along this path with respect to the boundary law $\lambda_{q_1}$.
	By the Ergodic Theorem for Markov chains (e.g. Theorem 4.4.1 in \cite{Br99}) we have
	\begin{equation*}
	\left(\frac{1}{n+1} \sum_{i=0}^{n} \chi_{\{\bar{k}\}}(\sigma'_{x_i}) \right)_{\bar{k} \in \Z_q} \stackrel{n \rightarrow \infty}{\rightarrow} (\alpha_{q_1}( \bar{k}))_{\bar{k} \in \Z_q} \quad \text{a.s.}
	\end{equation*} 
	where  $\alpha_{q_1}(\cdot)=\frac{\lambda_{q_1}(\cdot)^\frac{d+1}{d}}{\Vert \lambda_{q_1} \Vert_\frac{d+1}{d}^\frac{d+1}{d}}$ is the stationary distribution of the fuzzy chain.
	Now write $\sigma'_{x_i}=\sigma'_{x_0}+\sum_{j=1}^{i} \xi_{b_j}$ where $(\xi_{b_i})_{i=1, \ldots, n}$ denotes the increments of the  fuzzy chain along this path. 
	
	Fixing any $\bar{s} \in \Z_q$ and setting
	\begin{equation*}
	\tau^{q_1}_{x_i}:=\bar{s}+\sum_{j=1}^{i} \xi_{b_j}=\bar{s}-\sigma'_{x_0}+\sigma'_{x_i}=\Delta_{q_1}^{\bar{s}}+\sigma'_{x_i}
	\end{equation*}
	where $\Delta_{q_1}^{\bar{s}}=\bar{s}-\sigma'_{x_0}$ is a $\Z_q$-valued measurable function,
	we obtain a further Markov chain $(\tau^{q_1}_{x_i})_{i=1, \ldots, n+1}$.
	As $\chi_{\{k\}}(\tau^{q_1}_{x_i})=\chi_{\{k-\Delta_{q_1}^{\bar{s}}\}}(\sigma'_{x_i})$ 
	it follows that
	\begin{equation} \label{eq: ConvToBL}
	\left(\frac{1}{n+1} \sum_{i=0}^{n} \chi_{\{\bar{k}\}}(\tau^{q_1}_{x_i}) \right)_{\bar{k} \in \Z_q} \stackrel{n \rightarrow \infty}{\rightarrow} (\alpha_{q_1}(\bar{k}-\Delta_{q_1}^{\bar{s}}))_{\bar{k} \in \Z_q}=\left(\frac{\lambda_{q_1}(\bar{k}-\Delta_{q_1}^{\bar{s}})^\frac{d+1}{d}}{\Vert \lambda_{q_1}\Vert_\frac{d+1}{d}^\frac{d+1}{d}} \right)_{\bar{k} \in \Z_q}
	\end{equation} 
	almost surely.
	Applying the same procedure to the fuzzy chain associated to the boundary law $\lambda_{q_2}$ we arrive at 
	\begin{equation*}
	(\lambda_{q_2}(\bar{k}-\Delta_{q_2}^{\bar{s}}))_{\bar{k} \in \Z_q}=c(\lambda_{q_1}(\bar{k}-\Delta_{q_1}^{\bar{s}}))_{\bar{k} \in \Z_q} \quad \text{a.s.}
	\end{equation*} 
	where $c=\big(\frac{\Vert \lambda_{q_2}\Vert_\frac{d+1}{d}}{\Vert \lambda_{q_1}\Vert_\frac{d+1}{d}} \big)^\frac{d+1}{d}>0.$
	
	Hence there are some constant $c>0$ and some cyclic permutation $\rho: \Z_q \rightarrow \Z_q$ such that 
	\begin{equation*}
	\lambda_{q_2}(\bar{k})=c\lambda_{q_1}(\rho(\bar{k}))
	\end{equation*}
	which proves the case of $q_1=q_2$.
	
	In the general case let $\tilde{q}$ denote the least common multiple of $q_1$ and $q_2$. Both $\lambda_{q_1}$ and $\lambda_{q_2}$ are $\tilde{q}$-periodic, hence from the special case above we have that there are $c>0$ and some cyclic permutation $\rho: \Z_{\tilde{q}} \rightarrow \Z_{\tilde{q}}$ with $\lambda_{q_2}(\bar{k})=c\lambda_{q_1}(\rho(\bar{k}))$ for all $\bar{k} \in \Z_{\tilde{q}}$. By assumption $q_1$ and $q_2$ are the minimal periods of  $\lambda_{q_1}$ and $\lambda_{q_2}$, respectively, which implies $q_1=q_2$. This concludes the proof.
\end{proof} 
\begin{proof}[Proof of Corollary \ref{cor: Period}]
	First assume $\tilde{q}=q$. In that case from Equation \ref{eq: ConvToBL} we already know that the empirical distribution converges to a deterministic limit from which we can obtain the underlying boundary law by considering the $\frac{d}{d+1}$-th powers of its coordinates.
	
	In the case $\tilde{q} \neq q$ let $t=lcm(q,\tilde{q})$ denote the least common multiple. By $\lambda_t$ we then denote the $t$-periodic continuation of the boundary law $\lambda_q$. Note that for any $k \in \{0, \tilde{q}-1 \}$ we have the disjoint partition 
	\begin{equation}\label{eq: Partition}
	k+\tilde{q}\mathbb{Z}=\bigcup_{j=1}^{\frac{t}{\tilde{q}}}(k+(j-1)\tilde{q}+t\mathbb{Z}).
	\end{equation}
	
	Similar to the proof of Theorem \ref{thm: ident} consider a path
	
	$(b_1,b_2, \ldots, b_n)=((x_0,x_1), (x_1,x_2), \ldots, (x_{n-1},x_n))$  and let $\tau^t_{x_i}$ denote the mod-$t$ value of the total increment in $\Z_t$ along the path up to the vertex $x_i$ sampled by $\nu^{\lambda_q}$ and added up with some arbitrary starting value $s+t\Z$.  By \eqref{eq: Partition} it follows for any $k \in \{0, \ldots, \tilde{q}-1\}$
	\begin{equation}\label{eq: LimBL} 
	\frac{1}{n+1} \sum_{i=0}^{n} \chi_{\{k+\tilde{q}\mathbb{Z}\}}(\tau^{t}_{x_i}) \stackrel{n \rightarrow \infty}{\rightarrow} \sum_{j=1}^{\frac{t}{\tilde{q}}} \frac{\lambda_{t}(k+(j-1)\tilde{q}-\Delta_{t}^{\bar{s}}+t\mathbb{Z})^\frac{d+1}{d}}{\Vert \lambda_{t}\Vert_\frac{d+1}{d}^\frac{d+1}{d}}
	\end{equation}
	$\nu^{\lambda_q}$-a.s.,
	where $\Delta_{q_1}^{\bar{s}}=\bar{s}-\sigma'_{x_0}$ denotes the difference between the true starting value of the mod-$t$ fuzzy chain $(\sigma'_{x_i})_{n \geq 0}$ and $s+t\Z$.
	
	If $\tilde{q}$ is a multiple of $q$, then $t=\tilde{q}$ and we recover the values of boundary law $\lambda_t$ from the limit \eqref{eq: LimBL} for which $\nu^{\lambda_t}=\nu^{\lambda_q}$.
	
	If, on the other hand, the elementwise $\frac{d}{d+1}$-th powers of the limit \eqref{eq: LimBL} result in a $\tilde{q}$-periodic boundary law $\tilde{\lambda}_{\tilde{q}}$ for $Q$ and $\nu^{\tilde{\lambda}_{\tilde{q}}}=\nu^{\lambda_{q}}$ then $\tilde{q}$ must be a multiple of $q$. For, regarding $\tilde{\lambda}_{\tilde{q}}$ as a $t$-periodic boundary law for $Q$ it then must coincide with $\lambda_t$ up to a cyclic permutation of coordinates and $q$ was assumed to be the minimal period of $\lambda_q$.    
	
	Summarizing these arguments leads to the statement of the corollary.
	
\end{proof}

\begin{proof}[Proof of Proposition \ref{pro: GoodBinary}]
	A pair $(\gamma, \delta) \in (1, \infty) \times (0, \infty)$ lies in $G_2$ if and only if there exists an $\varepsilon>0$ such that the inequalities 
	\begin{equation} \label{eq: remBin}
	\begin{cases} \delta+\gamma \varepsilon^2  &\leq \varepsilon \\ L(\gamma, \delta)=4\gamma \varepsilon+4\delta \varepsilon^2&<1\end{cases}
	\end{equation} are satisfied. The first one is solved if and only if $\delta \leq \frac{1}{4\gamma}$, i.e., $1 \geq 4\gamma\delta$.
	For $\delta \leq \delta_0:=\frac{1}{4\gamma}$ we have the minimal positive solution \begin{equation*}
	\varepsilon(\delta,\gamma)=\frac{1}{2\gamma}(1 - \sqrt{1-4\delta \gamma}). 
	\end{equation*}
	Inserting this solution in the second inequality we obtain 
	\begin{equation*}
	L(\gamma, \delta)=2(1 - \sqrt{1-4\delta \gamma})+\frac{\delta}{\gamma^2}(1 - \sqrt{1-4\delta \gamma})^2.
	\end{equation*} 
	Writing $\sqrt{1-4\delta \gamma}=:a(\gamma, \delta)$ the inequality $L(\gamma, \delta)<1$ is equivalent to
	\begin{equation*}
	a^2(\gamma, \delta)\delta+\delta+\gamma^2<2a(\gamma, \delta)(\gamma^2+\delta).
	\end{equation*}
	Squaring both (positive) sites of this inequalilty and expanding it in the powers of $\delta$, this again is equivalent to
	\begin{equation*}
	16\gamma^2\delta^4+24\gamma^3\delta^2+\delta(16\gamma^5-4\gamma^2)-3\gamma^4 <0.
	\end{equation*}
	The equation $16\gamma^2\delta^4+24\gamma^3\delta^2+\delta(16\gamma^5-4\gamma^2)-3\gamma^4 =0$ is a quartic equation in $\delta$ with vanishing third-order coefficient for which we obtain, using MATHEMATICA, the unique positive solution and series expansion in $\gamma^{-1}$
	\begin{equation*}
	\begin{split}
	\delta(\gamma)
	&=\frac{1}{2}\sqrt{2\frac{\gamma^3-\frac{1}{4}}{\sqrt{(\gamma^3+\frac{1}{4})^\frac{2}{3}-\gamma}}-(\gamma^3+\frac{1}{4})^\frac{2}{3}-2\gamma}-\frac{1}{2}\sqrt{(\gamma^3+\frac{1}{4})^\frac{2}{3}-\gamma} \cr
	& = \frac{3}{16}\gamma^{-1}+O(\gamma^{-4}).
	\end{split}
	\end{equation*}	
	It is now easily verified that $0<\delta(\gamma) \leq \frac{1}{4\gamma}$ for all $\gamma \in (1,\infty)$. 
	
	Hence it follows that $(\gamma, \delta) \in (1, \infty) \times (0, \infty)$ lies in $G_2$ if and only if $\delta< \delta(\gamma)$.	
\end{proof}	

\begin{proof}[Proof of Theorem \ref{thm: LargeDegree}]
	By Theorem \ref{Coexthm: main} existence of the family $(\mu_i)_{i \in \Z^k}$ is guaranteed if the parameters $\delta_d(\beta)=\Vert Q_\beta \Vert_{d+1,\mathbb{Z}^k \setminus \{0\}}$ and $\gamma_d(\beta)=\Vert Q_\beta \Vert_{\frac{d+1}{2}, \mathbb{Z}^k} $ lie in the set $G_d$, i.e., if there exists an $\varepsilon>0$ such that the inequalities \ref{eq: Ball} and \ref{eq: L-const} are satisfied.
	For fixed $A>\frac{1}{v}$ let $d \geq 2$ be large enough such that $\Vert Q_{\beta_{A,d}} \Vert_{\frac{d+1}{2}}<\infty$.
	Then the set $M:=\{i \in \Z \setminus \{0\} \mid U(i)=v\}$ is finite.
	By monotonicity of $\delta_d(\beta)$ and $\gamma_d(\beta)$ in $\beta$ it suffices to show that $\delta_d(\beta_{A,d})$ and $\gamma_d(\beta_{A,d})$ lie in $G_d$.
	
	Using the short notations $q=q_{\beta_{A,d}}:=\exp(-\beta_{A,d} v)$ and $Q=Q_{\beta_{A,d}}$ we have
	\begin{equation*}
	\begin{split}
	\delta_d(\beta_{A,d}) 
	&=q \, \, ( \vert M \vert +\sum_{j \in \Z^k \setminus (M \cup \{0\})} (\tfrac{Q(j)}{q})^{d+1})^\frac{1}{d+1}  \cr
	&=\exp(- Av\frac{\log d}{d+1}) \, \left( \vert M \vert +\sum_{j \in \Z^k \setminus (M \cup \{0\})} (\tfrac{Q(j)}{q})^{d+1}\right)^\frac{1}{d+1}. 
	\end{split}
	\end{equation*} 
	By the assumption $\Vert Q_{\beta_{A,d}} \Vert_{\frac{d+1}{2}}<\infty$ the expression $\vert M \vert +\sum_{j \in \Z^k \setminus (M \cup \{0\})} (\tfrac{Q(j)}{q})^{d+1}$  is a finite number strictly decreasing with increasing $d$ and converging to $\vert M \vert$ by dominated convergence.
	Hence the factor $( \vert M \vert +\sum_{j \in \Z^k \setminus (M \cup \{0\})} (\tfrac{Q(j)}{q})^{d+1})^\frac{1}{d+1}$  is bounded from above by a positive constant to the power $\frac{1}{d+1}$. In contrast to this, the first factor is given by some constant to the power $\frac{\log d}{d+1}$.     
	Thus, for any fixed $\frac{1}{v}<A_1<A$ there is a $d_1 \in \mathbb{N}$ such that for all $d \geq d_1$ 
	\begin{equation*}
	\exp(- Av\frac{\log d}{d+1}) \, ( \vert M \vert +\sum_{j \in \Z^k \setminus (M \cup \{0\})} (\tfrac{Q(j)}{q})^{d+1})^\frac{1}{d+1} \leq \exp(- A_1v\frac{\log d}{d+1}).
	\end{equation*}
	Summarizing these observations and doing a first-order Taylor expansion we conclude that there are $\frac{1}{v}<A_2<A_1$ and $d_2 \geq d_1$ such that
	\begin{equation} \label{eq: delta_asymptotic}
	\delta_d(\beta_{A,d}) \leq \exp(-A_1 v \frac{\log d}{d+1}) \leq 1-A_2v  \frac{\log d}{d+1}
	\end{equation} 
	for all $d \geq d_2$.
	
	Similarly we have that the family $(\gamma_d(\beta_{A,d}))_{d \in \mathbb{N}}$ is bounded, hence in what follows we will simply write $\gamma$ for the upper bound of this family.
	
	In the next step we will give an upper bound on the minimal solution to the equation \eqref{eq: Ball} for large degrees. Fix any $\frac{1}{v}<A_3<A_2$ and set $\bar{\varepsilon}=\bar{\varepsilon}(A_3,d):=1-A_3 v \frac{\log d}{d+1}$. We will show that there is a $d_3 \geq d_2$ such that for $d \geq d_3$  the function $\bar{\varepsilon}$ satisfies the inequality $\gamma \bar{\varepsilon}^d+\delta_d(\beta_{A,d})  \leq \bar{\varepsilon}$.
	We have 
	\begin{equation*}
	\gamma \bar{\varepsilon}^d+\delta_d(\beta_{A,d})  \leq \gamma (1-A_3 v \frac{\log d}{d+1})^d+1-A_2 v \frac{\log d}{d+1}, 
	\end{equation*}
	so $\gamma \bar{\varepsilon}^d+\delta_d(\beta_{A,d})  \leq \bar{\varepsilon}$ is guaranteed by 
	\begin{equation*}
	\gamma (1-A_3 v \frac{\log d}{d+1})^d+1-A_2 v \frac{\log d}{d+1} \leq 1-A_3 v \frac{\log d}{d+1},
	\end{equation*}
	which is equivalent to  
	\begin{equation} \label{eq: LDA inv}
	\gamma (1-A_3 v \frac{\log d}{d+1})^d \leq (A_2-A_3) v \frac{\log d}{d+1}.
	\end{equation}
	Using first order Taylor-expansion the l.h.s. can be bounded as   
	\begin{equation} \label{eq: delta_bound}
	\begin{split}
	\gamma (1-A_3 v \frac{\log d}{d+1})^d&=\gamma \exp(d \log(1-A_3v \frac{\log d}{d+1})) \leq \gamma\exp(-d A_3v \frac{\log d}{d+1}) \cr 
	& <\gamma\exp(-\frac{d}{d+1}\log d)=\gamma(\frac{1}{d})^\frac{d}{d+1}.
	\end{split}
	\end{equation}
	Hence the quotient of the r.h.s. and the l.h.s. of equation \eqref{eq: LDA inv} is bounded from below by 
	\begin{equation*}
	\frac{(A_2-A_3)v}{\gamma}\frac{\log d}{d+1}d^{\frac{d}{d+1}}=\frac{(A_2-A_3)v}{\gamma}\frac{d}{d+1}\frac{\log d}{d^\frac{1}{d+1}} \stackrel{d \rightarrow \infty}{\rightarrow} \infty 
	\end{equation*} 
	which proves the existence of such a $d_3$ as the r.h.s of \eqref{eq: LDA inv} is strictly positive.  
	
	In the last step we insert $\bar{\varepsilon}$ into the l.h.s. of equation \eqref{eq: L-const} and obtain
	\begin{equation*}
	\begin{split}
	2d\gamma \bar{\varepsilon}^{d-1}+2d\delta_d(\beta_{A,d}) \bar{\varepsilon}^d \leq 2d \gamma(1-A_3 v \frac{\log d}{d+1})^{d-1}+2d(1-A_3 v \frac{\log d}{d+1})^{d+1}.
	\end{split}
	\end{equation*}
	Note that by Taylor expansion of the logarithm
	\begin{equation*}
	\begin{split}
	d(1-A_3 v \frac{\log d}{d+1})^{d-1}&=\exp(\log d+(d-1)\log(1-A_3 v \frac{\log d}{d+1})) \cr 
	&\leq \exp(\log d-\frac{d-1}{d+1}A_3v\log d) \cr 
	&=\exp((1-\frac{d-1}{d+1}A_3v)\log d).
	\end{split}
	\end{equation*} 
	Hence, by the assumption $A_3>\frac{1}{v}$ this converges to zero as $d$ tends to infinity.
	This shows that there is a $d_0 \geq d_3$ such that for all $d \geq d_0$ the upper bound $\bar{\varepsilon}$ satisfies the inequality \eqref{eq: L-const} concluding the proof of the first statement of the Theorem.
	
	To prove the second statement of the theorem, the asymptotic upper localization bound, let $d \geq d_0$. Recall that by Theorem \ref{Coexthm: main}
	\[\frac{\mu_i(\sigma_0 \neq i)}{\mu_i(\sigma_0=i)}\leq \left( \delta \frac{1+\delta\varepsilon(\gamma, \delta)^d}{1-\gamma\varepsilon(\gamma,\delta)^{d-1}} \right)^{d+1}.\]
	Now the inequality \eqref{eq: L-const} gives $-\gamma\varepsilon^{d-1}>\delta\varepsilon^d-\frac{1}{2d}$, hence
	\begin{equation*}
	\begin{split}
	\left( \delta \frac{1+\delta\varepsilon(\gamma, \delta)^d}{1-\gamma\varepsilon(\gamma,\delta)^{d-1}} \right)^{d+1} &< \left( \delta \frac{1+\delta\varepsilon(\gamma, \delta)^d}{1+\delta\varepsilon(\gamma, \delta)^d-\frac{1}{2d}} \right)^{d+1} \cr 
	&=\left( \delta \frac{1}{1-\frac{1}{2d(1+\delta \varepsilon^d)}} \right)^{d+1}.
	\end{split}
	\end{equation*}
	From inequality \eqref{eq: delta_asymptotic}  we know that $\delta \leq 1-\frac{\log d}{d}$ and an approximation similar to that of \eqref{eq: delta_bound} then gives $\delta^{d+1} \leq \frac{1}{d}$. 
	
	At last, the bound $(1-\frac{1}{2d(1+\delta \varepsilon^d)})^{d+1}>(1-\frac{1}{2d})^{d+1} \stackrel{d \rightarrow \infty}{\rightarrow} \exp(-\frac{1}{2})$ concludes 
	\begin{equation*}
	\frac{\mu_i(\sigma_0 \neq i)}{\mu_i(\sigma_0=i)} < C \,  \frac{1}{d} \, \stackrel{d \rightarrow \infty}{\rightarrow}0 
	\end{equation*}
	for some constant $C>0$. 
\end{proof}
Finally we give the proof of the well-definedness of the model, using a variation of 
the Young-inequality estimates seen before. 
\begin{proof}[Proof of Lemma \ref{lem: existenceModel}]
	Assume that $\Vert Q \Vert_{\frac{d+1}{2},\mathbb{Z}^k}<\infty$. 
	We claim that for any finite connected volume $\Lambda \subset V$ and any family of weights $(\lambda_a)_{a \in \partial \Lambda}$, where \\ $\lambda_a \in (0,\infty)^{\Z^k} \cap l_{\frac{d+1}{d}}(\mathbb{Z}^k)$, 
	the following auxiliary expression
	\begin{equation} \label{eq: ExistenceModel}
	\begin{split}
	\bar{Z_\Lambda}(\lambda_{\partial \Lambda})&:=\sum_{\omega_{\Lambda \cup \partial \Lambda}} \prod_{y \in \partial \Lambda} \lambda_{y}(\omega_y) \prod_{b \cap \Lambda \neq \emptyset} Q_b(\omega_b) \cr 
	&=\sum_{\omega_{\partial \Lambda}}\sum_{\eta_\Lambda}\prod_{y \in \partial \Lambda} \lambda_{y}(\omega_y) \prod_{b \cap \Lambda \neq \emptyset} Q_b((\eta_\Lambda\omega_{\partial \Lambda})_b) \cr 
	&=\sum_{\omega_{\partial \Lambda}}\prod_{y \in \partial \Lambda} \lambda_{y}(\omega_y)\left(\sum_{\eta_\Lambda}\prod_{b \cap \Lambda \neq \emptyset} Q_b((\eta_\Lambda\omega_{\partial \Lambda})_b) \right)\cr 
	&=\sum_{\omega_{\partial \Lambda}}\prod_{y \in \partial \Lambda} \lambda_{y}(\omega_y) Z_\Lambda(\omega_{\partial \Lambda})
	\end{split}
	\end{equation} 
	is finite.
	
	From this would follow the finiteness of the partition function $Z_\Lambda(\omega_{\partial \Lambda})$ for any boundary condition $\omega_{\partial \Lambda}$.
	
	The proof of the claim is done by induction on $\vert \Lambda \vert$. For convenience, we start with the induction step, i.e., we assume that the claim holds true for some finite connected volume $\Lambda \subset V$ and any family of weights $(\lambda_a)_{a \in \partial \Lambda}$ where $\lambda_a \in (0,\infty)^{\Z^k} \cap l_{\frac{d+1}{d}}(\mathbb{Z}^k)$. Let $v \in V \setminus \Lambda$ be some adjacent vertex and $(\lambda_a)_{a \in \partial (\Lambda \cup v)}$ be some family of weights in $(0,\infty)^{\Z^k} \cap l_{\frac{d+1}{d}}(\mathbb{Z}^k)$. As in Theorem \ref{Coex: BLMC} let $v_\Lambda$ denote the unique nearest-neighbor of $v$ in $\Lambda$. Then we have
	\begin{equation*}
	\bar{Z}_{\Lambda \cup v}(\lambda_{\partial \Lambda \setminus v},\lambda_{\partial v \setminus v_\Lambda})=\bar{Z}_{\Lambda}(\lambda_{\partial \Lambda \setminus v},\tilde{\lambda}_v),
	\end{equation*}
	where
	\begin{equation*}
	\tilde{\lambda}_v(i)=\prod_{x\in \partial v \setminus v_\Lambda}\sum_{j\in \mathbb{Z}^k}Q_{xv}(i - j)\lambda_{x}(j)
	\end{equation*}
	which follows by summing over spins for $x \in \partial v \setminus v_\Lambda$.
	To conclude the induction step $\bar{Z}_{\Lambda \cup v}(\lambda_{\partial \Lambda \setminus v},\lambda_{\partial v \setminus v_\Lambda})<\infty$ 
	it hence suffices to show that $\Vert \tilde{\lambda}_v \Vert_{\frac{d+1}{d}, \mathbb{Z}^k} < \infty$.
	
	First, generalized H\"older's inequality for $d$ factors, noting that\\ $\frac{d}{d+1}=\sum_{a=1}^d\frac{1}{d+1}$, gives
	\begin{equation*}
	\Vert \prod_{x\in \partial v \setminus v_\Lambda}\sum_{j\in \mathbb{Z}^k}Q_{xv}(\cdot - j)\lambda_{x}(j) \Vert_{\frac{d+1}{d},\mathbb{Z}^k} \leq  \prod_{x\in \partial v \setminus v_\Lambda}\Vert \, \sum_{j\in \mathbb{Z}^k}Q_{xv}(\cdot - j)\lambda_{x}(j) \, \Vert_{d+1, \mathbb{Z}^k}.
	\end{equation*}
	Now, from Young's inequality for convolutions with  $1+\frac{1}{d+1}=\frac{2}{d+1}+\frac{d}{d+1}$  it follows
	\begin{equation*}
	\prod_{x\in \partial v \setminus v_\Lambda}\Vert \, \sum_{j\in \mathbb{Z}^k}Q_{xv}(\cdot - j)\lambda_{x}(j) \, \Vert_{d+1, \mathbb{Z}^k} \leq \prod_{x\in \partial v \setminus v_\Lambda}\Vert Q_{xv} \Vert_{\frac{d+1}{2}, \mathbb{Z}^k} \Vert \lambda_{x} \Vert_{\frac{d+1}{d},\mathbb{Z}^k}<  \infty
	\end{equation*}
	which concludes the induction step.
	It remains to prove the initial step. In case of a single-element volume $\Lambda=\{v\}$ equation \eqref{eq: ExistenceModel} simply reads
	\begin{equation*}
	\begin{split}
	&\bar{Z}_{\{v\}}(\lambda_{\partial v}) \cr 
	&= \sum_{i \in \mathbb{Z}^k} \sum_{(j_1, \ldots, j_{d+1}) \in (\mathbb{Z}^k)^{d+1}}\prod_{y \in \partial v}Q_{yv}(i-j_y)\lambda_y(j_y) \cr 
	&=\sum_{i \in \mathbb{Z}^k}\prod_{y \in \partial v}\sum_{j \in \mathbb{Z}^k}Q_{yv}(i-j)\lambda_y(j)=\Vert \prod_{y \in \partial v}\sum_{j \in \mathbb{Z}^k}Q_{yv}(\cdot-j)\lambda_y(j) \Vert_{1, \mathbb{Z}^k}.
	\end{split}
	\end{equation*} 
	Again, H\"older's inequality applied with $1=\sum_{a=1}^{d+1}\frac{1}{d+1}$ 
	gives 
	\begin{equation*} 
	\Vert \prod_{y \in \partial v}\sum_{j \in \mathbb{Z}^k}Q_{yv}(\cdot-j)\lambda_y(j) \Vert_{1, \mathbb{Z}^k} \leq  \prod_{y \in \partial v}\Vert \, \sum_{j\in \mathbb{Z}^k}Q_{yv}(\cdot - j)\lambda_{y}(j) \, \Vert_{d+1, \mathbb{Z}^k},
	\end{equation*}
	and the rest of the proof follows from applying Young's inequality as above.	
\end{proof}

\printbibliography
\end{document}